\documentclass{article}
\usepackage{graphicx}
\usepackage{authblk}
\usepackage{amsmath, amsthm}
\usepackage{comment}
\usepackage{eulerpx}
\usepackage[T1]{fontenc}

\usepackage{bm}

\newtheorem{theorem}{Theorem}
\newtheorem{proposition}{Proposition}
\newtheorem{example}{Example}

\usepackage[overload]{empheq}

\usepackage{graphicx, grffile}
\graphicspath{{images/}}
\usepackage[a4paper,left=3cm,right=3cm]{geometry}
\usepackage{subfig}

\usepackage{tikz}
\usetikzlibrary{}

\usepackage{tabularx}

\usepackage{algorithm2e}

\usepackage[colorlinks, citecolor=black, linkcolor=black, urlcolor=black, linktoc=all]{hyperref}
\usepackage{cleveref}
\usepackage[numbered]{bookmark} 

\newcommand{\abs}[1]{\left|#1\right|}
\newcommand{\norm}[1]{\left\|#1\right\|}

\newtheorem{hypothesis}{Hypothesis}
\newtheorem{lemma}[subsection]{Lemma}

\newtheorem{remark}[subsection]{Remark}

\begin{document}

\title{The order of convergence in the averaging principle for slow-fast systems of stochastic evolution equations in Hilbert spaces}
\author{Filippo de Feo \\ filippo.defeo@polimi.it}

\affil{Department of Mathematics, Politecnico di Milano, \\ Piazza Leonardo da Vinci 32, 20133 Milano,
Italy }

\maketitle
\begin{abstract}
In this work we are concerned with the study of the strong order of convergence in the averaging principle for slow-fast  systems of stochastic evolution equations in Hilbert spaces with additive noise. In particular the stochastic perturbations are general Wiener processes, i.e their covariance operators are allowed to be not trace class. We prove that the slow component converges strongly to the averaged one with order of convergence $1/2$ which is known to be optimal. Moreover we apply this result to a slow-fast stochastic reaction diffusion system where the stochastic perturbation is given by a white noise both in time and space.
\end{abstract}
\textbf{Declarations of interest:} none
\section{Introduction}
Consider the following slow-fast system of abstract stochastic evolution equations  with additive noise
\begin{equation}\label{eq:system_introduction}
\begin{cases}
dU_t^{\epsilon}=A_1U_t^{\epsilon} dt+F(U_t^{\epsilon},V_t^{\epsilon})dt+dW^{Q_1}_t, \\
U_0^{\epsilon}=u \in  H, \\
dV_t^{\epsilon}=\epsilon^{-1}A_2V_t^{\epsilon}dt+\epsilon^{-1} G(V_t^{\epsilon})dt+\epsilon^{-1/2}dW^{Q_2}_t ,\\
V_0^{\epsilon}=v \in K,
\end{cases}
\end{equation}
where $\epsilon>0$ is a small parameter representing the ratio of time scales beween the slow component of the system $U^\epsilon$ and the fast one $V^\epsilon$. Here $H,K$ are Hilbert spaces, $A_1,A_2$ are unbounded linear operators on $H,K$ respectively and $W^{Q_1}, W^{Q_2}$ are Wiener processes on $H,K$ respectively.\\
Slow-fast systems are very used in applications since it is very natural for real-world systems to present very different time-scales. We refer the reader for example to \cite{friedlin} for applications to physics,  \cite{rawlings} to chemistry, \cite{wang} to neurophysiology, \cite {bardi}, \cite{defeo_nonautonomous}, \cite{papanicolaou}, \cite{papanicolaou2} to mathematical finance (see also \cite{defeo_thesis} for a slightly different financial model) and the references therein.\\
A natural idea is then to study the behaviour of the system when $\epsilon \to 0$. In particular under certain hypotheses it is known that the slow component $U^\epsilon$ converges to the solution $U$ of the so called averaged equation
\begin{equation*}
\begin{cases}
dU_t=A_1U_t dt+\overline{F}(U_t)dt+dW^{Q_1}_t, \\
U_0=u ,
\end{cases}
\end{equation*}
where $$\overline{F}(u)= \int_K F(u,v) \mu(dv),$$ 
and $\mu$ is the invariant measure related to the fast motion, i.e.
\begin{equation*}
\begin{cases}
dv_t^{v}=A_2v_t^{v} dt+ G(v_t^{v})dt+d w^{Q_2}_t ,\\
v_0^{v}=v \in K.
\end{cases}
\end{equation*}
Note that the equation for $U$ is uncoupled from $V^\epsilon$. This fact is known as averaging principle and it is fundamental in applications since $U$ captures the effective dynamic of $U^\epsilon$ (which is usually the most interesting variable in applications) and it is then a rigorous dimensionality reduction of the original system.\\
The first general result for the averaging principle for finite-dimensional stochastic differential equations can be found in \cite{Khasminskii}. For generalizations and improvements see  \cite{defeo_nonautonomous}, \cite{golec}, \cite{golec2}, \cite{friedlin}, \cite{Wei_liu}, \cite{veretnikov}, \cite{xu} and the references therein. It is important to mention that the drift of the fast equation is allowed to depend also on the slow component (i.e. fully coupled system) and the stochastic perturbations of the slow and fast equations are allowed to be multiplicative, i.e. the diffusion coefficient of the slow equation can depend on both the slow and fast variables. Moreover when the diffusion coefficient of the slow equation is independent of the fast variable then a strong convergence in probability is obtained. Otherwise only a weak convergence can be proved.

The averaging principle for infinite dimensional systems follows more delicated arguments: for this we refer to \cite{cerrai_friedlin}, \cite{Cerrai}, \cite{Cerrai_polynomial}, \cite{Cerrai_lunardi},   \cite{dong}, \cite{fu}, \cite{fu_wave_2015}, \cite{uda}   and the references therein. Also for infinite dimensional systems the previous comment about the dependence of the diffusion coefficients holds. See also \cite{tessitore_backward}, \cite{tessitore_backward2}, \cite{swiech_slow_fast_control} for optimal control problems of slow-fast systems in infinite dimension.

However for numerical applications it is very important to know the speed of convergence for which $U^\epsilon \to U$, e.g. see \cite{brehier2013}, \cite{WEINAN}. For the study of the order of convergence for finite dimensional systems we refer to  \cite{givon_rate2007},  \cite{khasminskii_rate}, \cite{di_liu},  \cite{rokner_averaging_order}, \cite{zhang_2018} and the references therein. It is important to mention that the order of convergence can be studied in two ways: in the strong sense and in the weak sense. Moreover the optimal order for the strong and weak convergence are known to be $1/2$ and $1$ respectively.

Recently the problem of estimating the order of convergence in the averaging principle for infinite dimensional systems systems is being addressed by researchers:\\ 
in \cite{brehier2020} the author (generalizing his previous work \cite{brehier2012}) considers a slow-fast stochastic reaction diffusion system  with additive noise. Both the weak and  strong orders of convergence are obtained: in particular under strong regularity of the noise (it is for example assumed that the covariance operator is trace class but for the precise statement see \cite{brehier2020}) it is proved that the strong order of convergence is $1/2$ and the weak order is $1$ with both orders being optimal. Instead under more general assumptions on the noise only weaker orders of convergence are obtained for both the strong and weak convergence.\\
In \cite{wang} a $1$-dimensional fully coupled reaction-diffusion system is considered and the strong order of convergence is proved to be $1/2$ under very strong assumptions on the covariance operators of the noises, i.e. $Tr (\Delta^{1/2} Q_i)< \infty$, where $\Delta$ is the Laplacian.\\
In \cite{rokner_averaging_devitiation_spdes} the strong order of convergence for a fully coupled slow-fast stochastic system is studied. Here it is assumed that the the covariance operators of the noises are trace-class and moreover that $Tr (-A_1 Q_1) < \infty$.\\
See also \cite{fu_order_2018}  where the weak order of convergence for a stochastic wave equation with fast oscillation given by a fast reaction-diffusion stochastic system is proved to be $1$. Also here it assumed $Tr ( Q_i)< \infty$.\\
Indeed in all these papers the case $Tr ( Q_i)= \infty$, which is very important for applications  as it happens very naturally for example when the stochastic perturbation is a white noise i.e. $Q_i=I$, can't be treated.\\
In this manuscript we are then interested in studying the strong order of convergence for the slow-fast infinite-dimensional system of stochastic evolution equations \eqref{eq:system_introduction} where $W^{Q_1}, W^{Q_2}$ are general Wiener processes on $H,K$ respectively  with covariance operators $Q_1,Q_2$ with $Tr(Q_1)=+ \infty$, $Tr(Q_2) = +\infty$ possibly.  Under some hypotheses, see Hypotheses \ref{hp:operator1}, \ref{hp:operator2}, \ref{hp:coefficients}, \ref{hp:hilbert_schmidt}, \ref{hp:covariance_inverse} below, we prove that the strong order of convergence is $1/2$ which is known to be optimal. In particular we show in Theorem \ref{th:main_theorem} that
\begin{equation*}
\mathbb{E}\left [ \sup_{t\in [0,T]} \abs{U^\epsilon_t-U_t}^2 \right] \leq C \epsilon 
\end{equation*}
where $U_t$ is the solution of the averaged equation. 
Notice that this result is much stronger than \cite{brehier2020}, \cite{rokner_averaging_devitiation_spdes} where $\sup_{t \in [0,T]}$ is outside the expectation. 

The key tool in the proof of Theorem \ref{th:main_theorem} is  Proposition \ref{crucial_lemma_2}. The proof of this proposition is  based on  a technical result, i.e. Lemma \ref{lemma:crucial_lemma}, which is inspired by \cite{Cerrai_normal_dev}, and it is a consequence of the mixing properties of the fast motion, i.e. Lemmas \ref{lemma:mixing_1}, \ref{lemma:mixing_2}, \ref{lemma:mixing_3}.  We recall that \cite{Cerrai_normal_dev} studies the normal deviations, i.e. the weak convergence of $Z^\epsilon:=(U^\epsilon-U)/ \sqrt \epsilon$, when the equation for the slow component has no stochastic perturbation ($Q_1=0$).

Finally we discuss an application of our theory to a $1$-dimensional slow-fast stochastic reaction diffusion system where the stochastic perturbation is given by a white noise both in time and space which to the best of our knowledge, as said before,  can not be treated by the existing literature. 

The paper is organized as follows:
in section \ref{sec:setup} we introduce the problem in a formal way and we state the assumptions that we will use.
In section \ref{sec:estimates} we prove some a-priori estimates. In section \ref{sec:fast_motion} we prove some results related to the fast motion. In section \ref{sec:averaged_eq} we study the well posedness of the averaged equation. In section \ref{sec:preliminary_results} we prove some preliminary results.
In section \ref{sec:order} we prove that the order of convergence is $1/2$ and we give an application of our theory.

\section{Setup and assumptions}
\label{sec:setup}
In this section we define the notation and the assumptions for the rest of the paper.

$H,K$ will be Hilbert spaces with scalar products $\langle \cdot ,\cdot \rangle_H, \langle \cdot ,\cdot \rangle_K$ and $|\cdot|_H, |\cdot|_K$ the induced norms.\\
$B_B(H)$ will denote the space of bounded functions $\phi \colon H \to \mathbb{R}$ with the sup norm $|\cdot|_{H,\infty}$. \\
$Lip(H)$ will denote the set of Lipschitz functions $\phi \colon H \to \mathbb{R}$ and set 
\begin{equation*}
[\phi]_{\text {H,Lip }}=\sup _{x \neq y} \frac{|\phi(x)-\phi(y)|}{|x-y|_H}.
\end{equation*}
$\mathcal{L}(H)$ will denote the space of linear bounded operators from $H$ to $H$, endowed with the operator norm $$\norm{L}_H=\sup_{|x|_H=1}|Lx|_H.$$
Next denote by $\mathcal{L}_2(H)$ the space of Hilbert-Schmidt operators endowed with the norm
$$\norm {L}_{\mathcal{L}_2(H)}=(Tr(L^*L))^{1/2}.$$
The analogous spaces $B_B(K),Lip(K),\mathcal{L}(K),\mathcal{L}_2(K)$ are defined for the Hilbert space $K$ with the corresponding norms $|\cdot|_{K,\infty},[\cdot]_{\text {K,Lip }},\norm{\cdot}_K,\norm{\cdot}_{\mathcal{L}_2(K)}$.\\
In order to simplify the notation we will omit the subscripts $K$ and $H$ in the various norms when no confusion is possible.\\
$\mathcal{B}(H)$ and $\mathcal{B}(K)$ will denote the Borel sigma-algebra in $H$ and $K$ respectively.\\
Consider now the following infinite dimensional system for $0 \leq t \leq T < \infty$
\begin{equation}\label{eq:abstract_system}
\begin{cases}
dU_t^{\epsilon}=A_1U_t^{\epsilon} dt +F(U_t^{\epsilon},V_t^{\epsilon})dt+dW^{Q_1}_t ,\\
U_0^{\epsilon}=u \in H , \\
dV_t^{\epsilon}=\epsilon^{-1}A_2V_t^{\epsilon} dt+\epsilon^{-1} G(V_t^{\epsilon})dt+\epsilon^{-1/2}dW^{Q_2}_t, \\
V_0^{\epsilon}=v \in K,
\end{cases}
\end{equation}
 where
\begin{itemize}
\item $\epsilon>0$ is a parameter,
\item $A_1 \colon D(A_1) \subset H \to H$, $A_2 \colon D(A_2) \subset K \to K$ are linear operators,
\item $F \colon H\times K\rightarrow  H$, $G \colon  K\rightarrow  K$,
\item $W^{Q_1}, W^{Q_2}$ are independent cylindrical Wiener processes on $H$, $K$ respectively with covariance operator $Q_1,Q_2$ respectively and they are defined on some probability space $(\Omega,\mathcal{F},\mathbb{P})$ with a normal filtration $\mathcal{F}_t$, $t \geq0 $.
\end{itemize}
We now state the assumptions that we will use throughout the work:
\begin{hypothesis}\label{hp:operator1}
$A_1 \colon D(A_1) \subset H \to H$ is a linear operator generator of an analytical semigroup $e^{A_1 t}$ on $H$, $t \geq0 $.\\ Moreover there exist an orthonormal basis $\{e_k\}_{k \in \mathbb{N}}$ of $H$ and $\{\alpha_k\}_{k \in \mathbb{N}} \subset (0, +\infty)$ such that
$$A_1 e_k = -\alpha_k e_k.$$ 
Moreover we assume that there exist $\zeta>0,n \geq 2$ integer and $ 1/(2n)< \beta < 1/3$ such that
$$\sum_{k=1}^{\infty} \alpha_{k}^{-\zeta} < +\infty$$
and
$$\sum_{k=1}^{\infty} \alpha_{k}^{n(\zeta+2\beta-1)-\zeta}< +\infty.$$
\end{hypothesis}
\begin{remark}\label{rem:laplacian_hp_1}
Hypothesis \ref{hp:operator1} is necessary in the proof of Proposition \ref{crucial_lemma_2}. It holds for example when $A_1= \Delta$ is the Laplacian on $[0,L]$. Indeed in this case it is well known that $\alpha_k = C k^2$ and then the two series converge for example by choosing  $\zeta=\frac{3}{5}$, $n=3$, $\beta=\frac{1}{5}$.
\end{remark}
From Hypothesis \ref{hp:operator1} we have the following spectral representation:
\begin{equation}\label{eq:spectral_representation_A_1}
e^{At}x=\sum_{k=1}^{\infty}e^{-\alpha_k t}\langle x,e_k \rangle e_k ,\quad \forall t >0.
\end{equation}
Moreover we can define the fractional powers of $-A_1$ denoted by  $(-A_1)^{\theta}$ for $\theta  \geq 0$ with domain $D((-A_1)^\theta)$. We will denote by $|\cdot|_{\theta}$ the norm $|\cdot|_{D((-A_1)^\theta)}=|(-A_1)^\theta \cdot|$.\\
The following standard results holds, e.g. see \cite[Chapter 2, Theorem 6.13]{pazy},
\begin{equation}\label{eq:bound_-A_exp At}
\| (-A_1)^\theta e^{At} \| \leq C_\theta t^{-\theta} e^{-\nu t} , \quad \forall \theta \geq 0, \textcolor{red}{ t > 0}, 
\end{equation}
for some $\nu >0$ and
\begin{equation}\label{eq:holder_continuity_semigroup}
|e^{A_1t}x-x| \leq C_\theta t^{\theta}|x|_{\theta}, \quad \forall x \in D\left ((-A_1)^\theta \right ), 0< \theta \leq 1 , t \geq 0.
\end{equation}
\begin{hypothesis}\label{hp:operator2}
$A_2 \colon D(A_2) \subset K \to K$ is a linear operator generator of a $\mathcal{C}_0$-semigroup $e^{A_2 t}$ on $K$, $t \geq0 $.\\ Moreover there exists $\lambda>0$ such that
\begin{equation}
\norm{e^{A_2 t }}\leq e^{-\lambda t},
\end{equation}
for every $t\geq 0$.
\end{hypothesis}
\begin{hypothesis}\label{hp:coefficients}
There exist $L_F, L_G>0$ such that
\begin{align*}
&|F(x_2,y_2)-F(x_1,y_1)|\leq L_F (|x_2-x_1|+|y_2-y_1| ),\\
&|G(y_2)-G(y_1)|\leq L_G |y_2-y_1| ,
\end{align*}
for every $x_1,x_2 \in H, y_1,y_2 \in K$.\\
Moreover we assume that $$L_G < \lambda.$$
\end{hypothesis}
We remark that this implies that $A_2+G(\cdot)$ is strongly dissipative, i.e. set 
\begin{equation}\label{eq:definition_delta}
\delta=\frac{\lambda-L_G}{2}>0 ,
\end{equation}
then it holds:
\begin{equation}\label{eq:strong_dissipativity_G}
\langle A_2 (y_2-y_1) + G(y_2) -G(y_1) , y_2-y_1\rangle  \leq -2\delta|y_2-y_1|^2 ,
\end{equation}
for every  $y_1,y_2 \in D(A_2)$.\\
\begin{hypothesis}\label{hp:hilbert_schmidt}
There exist $C>0$, $\gamma \in (0,1/2)$ such that
\begin{align*}
&\norm {e^{A_1t}Q_1^{1/2} }_{\mathcal{L}_2}\leq C (t \wedge 1)^{-\gamma}, \\
&\norm { e^{A_2t}Q_2^{1/2} }_{\mathcal{L}_2}\leq C (t \wedge 1)^{-\gamma} ,
\end{align*}
for every $t>0$.
\end{hypothesis}
\begin{hypothesis}\label{hp:covariance_inverse}
Assume that $Q_2$ is invertible (with inverse $Q_2^{-1} \in \mathcal L(K)$).
\end{hypothesis}
\begin{remark}\label{rem:covariances_laplacian}
Hypotheses \ref{hp:hilbert_schmidt}, \ref{hp:covariance_inverse} hold for example choosing $A_1= A_2=A$ to be the Laplacian on $[0,L]$ and $Q_1=Q_2=Q=I$.\\
Indeed Hypothesis \ref{hp:covariance_inverse} is immediately satisfied.
Moreover by setting $H=K=L^2([0,L])$, we have that $A e_k =-C k^2 e_k$ for some orthonormal basis $\{ e_k\}$ of eigenvectors of $A$. Thus, by the spectral representation \eqref{eq:spectral_representation_A_1} with $\alpha_k=C k^2$, we have:
\begin{small}
\begin{align*}
\norm {e^{A t}Q^{1/2} }_{\mathcal{L}_2}^2&=\norm {e^{A t} }_{\mathcal{L}_2}^2 =\sum_{h=1}^\infty |e^{At}e_h|^2
=\sum_{h=1}^{\infty} e^{-2Ch^2t} \leq \int_{0}^{\infty} e^{-2Ch^2t} \textcolor{red}{dh}=\frac{1}{\sqrt{2Ct}}\int_{0}^{\infty} e^{-y^2} dy = C t ^{-1/2},
\end{align*}
\end{small}
where the inequality follows since $\forall t > 0$ the function $h \mapsto e^{-Ch^2 t}$ is non-increasing. This shows that Hypothesis \ref{hp:hilbert_schmidt} holds with $\gamma=1/4.$
\end{remark}
\begin{proposition}
Let $u \in H, v \in K$, under Hypotheses \ref{hp:operator1}, \ref{hp:operator2}, \ref{hp:coefficients}, \ref{hp:hilbert_schmidt}, \ref{hp:covariance_inverse} there exists a unique mild solution of \eqref{eq:abstract_system} given by 
\begin{equation}\label{eq:mild_solution}
\begin{cases}
U_t^{\epsilon}=e^{A_1t}u +\int_0^te^{A_1(t-s)} F(U_s^{\epsilon},V_s^{\epsilon})ds +\int_0^te^{A_1(t-s)} dW^{Q_1}_s, \\
V_t^{\epsilon}=e^{A_1t}v +\int_0^te^{A_2(t-s)} G(V_s^{\epsilon})ds +\int_0^te^{A_2(t-s)} dW^{Q_2}_s ,
\end{cases}
\end{equation}
for every $t \in [0,T]$.
\end{proposition}
\begin{proof}
See e.g. \cite{daprato_stochastic_eq}.
\end{proof}
In the sequel we will always assume that Hypotheses  \ref{hp:operator1}, \ref{hp:operator2}, \ref{hp:coefficients}, \ref{hp:hilbert_schmidt}, \ref{hp:covariance_inverse} hold. Moreover  $C>0$ will denote a generic constant independent of $\epsilon$ which may change from line to line. 
\section{A priori estimates}\label{sec:estimates}
In this section we prove some classical a-priori estimates for the slow and fast components.\\
In the following for every $t \geq 0$ denote by
 $$\Gamma^{1}_t=\int_0^t e^{A_1 (t-s)} dW_s^{Q_1}$$
and $$\Gamma^{2 \epsilon}_t= \frac{1}{\sqrt{\epsilon}}\int_0^t e^{A_2 (t-s)/ \epsilon} dW_s^{Q_2}$$
the stochastic convolutions.\\
First we prove some estimates related to $\Gamma^{1}_t$ and $\Gamma^{2 \epsilon}_t$.
\begin{lemma}\label{lemma:estimate_stoch_convolution}
\begin{equation*}
\mathbb{E}\left [ \sup_{t\in [0,T]} \abs{\Gamma^{1}_t}_{\theta}^p \right] < +\infty,
\end{equation*}
for every $0 \leq \theta < 1/2-\gamma$, $p \geq 1$.
\end{lemma}
\begin{proof}
Fix $0 <\eta < 1/2$, by the factorization method, e.g. see \cite[Chapter 5, Section 3]{daprato_stochastic_eq}, we have:
$$\Gamma^{1}_t=C \int_0^t e^{A_1(t - \rho)}(t-\rho)^{\eta-1}Y_\rho d\rho$$
where 
$$Y_\rho=\int_0^\rho e^{A_1 (\rho-s)} (\rho-s)^{-\eta}dW_s^{Q_1}.$$
Now fix $0 \leq \theta < 1/2-\gamma$, then by Holder's inequality:
\begin{align}\label{eq:proof_lemma_stoch_conv_gamma1}
\mathbb{E} \left[\sup_{t \leq T}  \left |\Gamma^{1}_t \right|^p_\theta \right] \leq C \int_0^T \mathbb{E} |Y_\rho|^p_\theta d\rho,
\end{align}
for every $p > 1/\eta>2$.\\
Now we estimate $\mathbb{E} |Y_\rho|^p_\theta$; since $Y_\rho$ is a Gaussian random variable, by Ito's isometry, \eqref{eq:bound_-A_exp At} and Hypothesis \ref{hp:hilbert_schmidt} we have:
\begin{align*}
\mathbb{E} |Y_\rho|^p_\theta  & \leq C \left( \mathbb E |Y_\rho|_\theta^2 \right)^{p/2} =  C \left ( \int_0^\rho \norm{(-A_1)^\theta e^{A_1 (\rho-s)} (\rho-s)^{-\eta} Q_1^{1/2}}^2_{\mathcal{L}_2} ds \right)^{p/2} \\
& \leq C \left ( \int_0^\rho  (\rho-s)^{-2\eta} \norm{(-A_1)^\theta e^{A_1 (\rho-s)/2}}^2  \norm{ e^{A_1 (\rho-s)/2} Q_1^{1/2}}^2_{\mathcal{L}_2} ds \right)^{p/2} \\
& \leq C \left ( \int_0^\rho  (\rho-s)^{-2(\eta+\theta+\gamma)}  ds \right)^{p/2} < C,
\end{align*}
for every $\rho \leq T$ and $\theta, \eta$ such that $0 \leq \theta + \eta < 1/2 -\gamma$.\\
Next inserting the last inequality in \eqref{eq:proof_lemma_stoch_conv_gamma1} and recalling that $p > 1/\eta$, which yields $ p > (1/2 - \gamma - \theta)^{-1}$
 we obtain the thesis of the Lemma for $0 \leq \theta < 1/2-\gamma$, $p > (1/2 - \gamma - \theta)^{-1}>2$.\\
 Finally by Holder's inequality we have the thesis of the Lemma.
\end{proof}
\begin{lemma}\label{lemma:estimate_stoch_convolution_2}
Let $p\geq 2$, then there exists $C=C(p)>0$ such that
\begin{equation*}
\sup_{t>0} \mathbb{E}\left [ \abs{\Gamma^{2  \epsilon}_t}^p \right] \leq C,
\end{equation*}
for every $\epsilon>0$.
\end{lemma}
\begin{proof}
For $t>0$, by \cite[Theorem 4.36]{daprato_stochastic_eq} and by our hypotheses we have
\begin{align*}
\mathbb{E}\left [ \abs{\Gamma^{2  \epsilon}_t}^p \right] & \leq C \left( \int_0^t \frac{1}{\epsilon}  \norm{e^{A_2(t-s)/\epsilon} Q_2^{1/2}}_{\mathcal{L}_2}^2ds \right)^{p/2}\\
&= C  \left( \int_0^{t/ \epsilon} \norm{e^{A_2 \rho} Q_2^{1/2}}_{\mathcal{L}_2}^2 d \rho \right)^{p/2}\\
& \leq C \left( \int_0^{t/ \epsilon} \norm{e^{A_2 \rho/2} }^2 \norm{e^{A_2 \rho/2} Q_2^{1/2}}_{\mathcal{L}_2}^2 d \rho  \right)^{p/2} \\
& \leq C \left( \int_0^{t/ \epsilon}  e^{-\lambda \rho}  \left(\frac{\rho}{2} \wedge 1 \right )^{-2 \gamma} d \rho \right)^{p/2} \\
& \leq C \left( \int_0^{\infty}  e^{-\lambda \rho}  \left(\frac{\rho}{2} \wedge 1 \right )^{-2\gamma} d \rho \right)^{p/2}  \leq C ,
\end{align*}
so that the thesis is proved.
\end{proof}
\begin{lemma}\label{lemma:a_priori_estimates}
Let $p \geq 2$ then there exists $C=C(T,p)>0$ such that
\begin{align}\label{eq:a_priori_estimate_U}
 \mathbb{E}\left [ \sup_{t\in [0,T]} \abs{U^\epsilon_t}^p \right] \leq C (1+|u|^p+|v|^p)
 \end{align}
 and
 \begin{align}\label{eq:a_priori_estimate_V}
\sup_{t\in [0,T]} \mathbb{E}\left [ \abs{V^\epsilon_t}^p \right] \leq C (1+|u|^p+|v|^p),
\end{align}
for every $\epsilon>0$.
\end{lemma}
\begin{proof}
Define
$$ \Lambda^{1\epsilon}_t=U_t^\epsilon-\Gamma^{1}_t,$$
so that
$$d \Lambda^{1\epsilon}_t= A_1 \Lambda^{1\epsilon}_t dt + F(\Lambda^{1\epsilon}_t+\Gamma^{1}_t , V^\epsilon_t)dt.$$
By Young's inequality and  Hypotheses \ref{hp:operator1},  \ref{hp:coefficients} we have:
\begin{align*}
\frac{1}{p} \frac{d}{d t}\left|\Lambda^{1\epsilon}_t \right|^{p}=
&\langle A_{1} \Lambda^{1\epsilon}_t, \Lambda^{1 \epsilon}_t \rangle \left|\Lambda^{1\epsilon}_t\right|^{p-2}+
\langle F(\Lambda^{1\epsilon}_t+\Gamma^{1}_t, V^{\epsilon}_t ) , \Lambda^{1 \epsilon}_t \rangle \left|\Lambda^{1\epsilon}_t\right|^{p-2}
\\
\leq & C\left|\Lambda^{1\epsilon}_t\right|^{p}+C\left|F(\Lambda^{1\epsilon}_t+\Gamma^{1}_t, V^{\epsilon}_t ) \right|^{p} \\
\leq & C\left|\Lambda^{1\epsilon}_t\right|^{p}+C\left(1+\left|\Gamma^{1 }_t\right|^{p}+\left|V^{\epsilon}_t\right|^{p}\right),
\end{align*}
for every $t \leq T$.\\
Then by the comparison Theorem we have 
\begin{equation}
\left |\Lambda^{1\epsilon}_t\right|^p \leq C |u|^p + C \int_0^t \left (1+\left |\Gamma^{1}_s\right|^p+  \left |V^{\varepsilon}_s\right|^p \right) ds,
\end{equation}
for every $t \leq T$.\\
Then by the definition of $\Lambda^{1\epsilon}$ and this last inequality it follows
\begin{equation}
\begin{aligned}
\left|U_t^\epsilon\right|^{p} 
\leq & C \left(1+|u|^{p}+\int_{0}^{t} \left (\left|\Gamma_s^1\right|^{p} +\left|V^{\epsilon}_s\right|^{p} \right) d s +\left|\Gamma_t^1\right|^{p} \right),
\end{aligned}
\end{equation}
for every $t \leq T$.\\
Now by Lemma \ref{lemma:estimate_stoch_convolution} we have:
\begin{equation}\label{eq:proof_estimate_for_u_with_integral_v}
\mathbb{E} \left[ \sup _{t \leq \tau}\left|U_t^\epsilon\right|^{p} \right] \leq C \left(1+|u|^{p}+\int_{0}^{\tau} \mathbb{E}\left|V^{\epsilon}_r\right|^{p} dr \right),
\end{equation}
for every $\tau \leq T$.\\
Now we proceed in a similar way to \cite[proof of Lemma 3.10]{tessitore_backward}, i.e. set 
$$ \Lambda^{2\epsilon}_t=e^{\frac{2 \delta}{\epsilon} t}\left (V_t^\epsilon-\Gamma^{2\epsilon}_t \right),$$
so that
$$d\Lambda^{2\epsilon}_t = \frac{2 \delta}{\epsilon} \Lambda^{2\epsilon}_t dt+ \frac{1}{\epsilon} A_2 \Lambda^{2\epsilon}_t  dt + \frac{1}{\epsilon} e^{\frac{2 \delta}{\epsilon} t} G(e^{-\frac{2 \delta}{\epsilon}}\Lambda^{2\epsilon}_t  + \Gamma^{2\epsilon}_t) dt.  $$
Now by \eqref{eq:strong_dissipativity_G} it follows
\begin{align*}
d |\Lambda^{2\epsilon}_t|^2 & = \frac{4 \delta}{\epsilon} |\Lambda^{2\epsilon}_t|^2 + \frac{2}{\epsilon} \langle A_2\Lambda^{2\epsilon}_t, \Lambda^{2\epsilon}_t \rangle dt + \frac{2}{\epsilon}  e^{2 \delta  t / \epsilon} \langle G(e^{-2 \delta t/ \epsilon} \Lambda^{2\epsilon}_t + \Gamma^{2\epsilon}_t), \Lambda^{2\epsilon}_t\rangle dt  \\
& \quad - \frac{2}{\epsilon}  e^{2 \delta  t / \epsilon} \langle G(\Gamma^{2\epsilon}_t), \Lambda^{2\epsilon}_t\rangle dt +  \frac{2}{\epsilon}  e^{2 \delta  t / \epsilon} \langle G(\Gamma^{2\epsilon}_t), \Lambda^{2\epsilon}_t\rangle dt \\
& \leq \frac{2}{\epsilon}  e^{2 \delta  t / \epsilon} \langle G( \Gamma^{2\epsilon}_t), \Lambda^{2\epsilon}_t\rangle dt .
\end{align*}
Now similarly to \cite[proof of Lemma 3.10]{tessitore_backward} fix $\theta>0$, differentiate $f(x)= \sqrt{x+ \theta }$ and use the previous inequality, then:
\begin{align*}
|\Lambda^{2\epsilon}_t |& \leq \sqrt{|\Lambda^{2\epsilon}_t |^2+ \theta } \\
&  \leq \sqrt{|v|^2+\theta}+\int_0^t \frac{1}{\epsilon} \frac{1}{\sqrt{|\Lambda^{2\epsilon}_s|^2+\theta}} e^{\frac{2 \delta}{\epsilon} s} |G(\Gamma^{2\epsilon}_s)||\Lambda^{2\epsilon}_s|  ds .
\end{align*}
Now by dominated convergence for $\theta \to 0$ we have:
$$|\Lambda^{2\epsilon}_t | \leq |v| + \int_0^t \frac{1}{\epsilon} e^{\frac{2 \delta}{\epsilon} s} |G( \Gamma^{2\epsilon}_s)|  ds,$$
so that by recalling the definition of $\Lambda^{2\epsilon}_t$ we have:
\begin{align*}|V_t^\epsilon| & \leq e^{-\frac{2 \delta}{\epsilon} t}  |v| + \int_0^t \frac{1}{\epsilon} e^{-\frac{2 \delta}{\epsilon} (t-s)} |G( \Gamma^{2\epsilon}_s)|  ds +| \Gamma^{2\epsilon}_t |\\
& = e^{-\frac{2 \delta}{\epsilon} t}  |v| + \int_0^{t/\epsilon} e^{-2 \delta (t / \epsilon - s)} |G( \Gamma^{2\epsilon}_{\epsilon s})|  ds +| \Gamma^{2\epsilon}_t |.
\end{align*}
Then, by Holder's inequality and Lemma \ref{lemma:estimate_stoch_convolution_2},  we  have:
\begin{align*}
\mathbb{E}|V_t^\epsilon|^p & \leq C  |v|^p + C \mathbb{E} \left | \int_0^{t / \epsilon} e^{-2 \delta (t / \epsilon - s)} |G( \Gamma^{2\epsilon}_{\epsilon s})|   ds \right |^p + C\mathbb{E} \left |  \Gamma^{2\epsilon}_t \right|^p \\
&  \leq C  |v|^p + C \left( \int_0^{t / \epsilon}  e^{-\frac{ 2\delta p}{ (p-1)} (t / \epsilon - s)} ds \right)^{p-1}  \int_0^{t / \epsilon}  e^{-2\delta p (t / \epsilon - s)} \mathbb{E}|G(\Gamma^{2\epsilon}_{\epsilon s})| ^p  ds+ C \mathbb{E} \left |  \Gamma^{2\epsilon}_t \right|^p \\
& = C  |v|^p + C   \left( \int_0^{t / \epsilon }   e^{-\frac{ 2\delta p}{ (p-1)} \sigma} d\sigma \right)^{p-1}  \int_0^{t } \frac{1}{\epsilon}  e^{-2\delta p (t  - s)/ \epsilon} \mathbb{E}|G(\Gamma^{2\epsilon}_{ s})| ^p ds +C  \mathbb{E}\left |  \Gamma^{2\epsilon}_t \right|^p \\
&\leq C  |v|^p + C   \left( \int_0^{t / \epsilon }   e^{-\frac{ 2\delta p}{ (p-1)} \sigma} d\sigma \right)^{p-1}  \int_0^{t } \frac{1}{\epsilon}  e^{-2\delta p (t  - s)/ \epsilon} (1+ \mathbb{E}|\Gamma^{2\epsilon}_s|^p ) ds +C  \mathbb{E}\left |  \Gamma^{2\epsilon}_t \right|^p \\
&\leq C  |v|^p + C    \left( \int_0^{t / \epsilon }   e^{-\frac{ 2\delta p}{ (p-1)} \sigma} d\sigma \right)^{p-1}   \int_0^{t } \frac{1}{\epsilon}  e^{-2\delta p (t  - s)/ \epsilon}  ds +C   \\
& \leq C(1+ |v|^p).
\end{align*}
This proves \eqref{eq:a_priori_estimate_V}.

Finally inserting \eqref{eq:a_priori_estimate_V} into \eqref{eq:proof_estimate_for_u_with_integral_v} we have \eqref{eq:a_priori_estimate_U}.
 \end{proof} 
\begin{lemma}\label{lemma:estimate_U_norm_alpha}
Let $0 < \alpha < 1/2 - \gamma$, $u \in D((-A_1)^\alpha)$, $v \in K$, then there exists $C=C(T,\alpha)>0$ such that
\begin{align*}
 \mathbb{E}\left [ \sup_{t\in [0,T]} \abs{U^\epsilon_t}_{\alpha}^2 \right] \leq C (1+|u|_\alpha^2+|v|^2),
 \end{align*}
 for every $\epsilon>0$.
 \end{lemma}
 \begin{proof}
 Consider for $t \leq T$ 
\begin{align}\label{eq:mild_U}
U_t^{\epsilon}=e^{A_1t}u +\int_0^te^{A_1(t-s)} F(U_s^{\epsilon},V_s^{\epsilon})ds +\int_0^te^{A_1(t-s)} dW^{Q_1}_s.
\end{align}
 First as $u \in D((-A_1)^\alpha)$ we have:
 $$\sup_{t \leq T}|e^{A_1 t}u|_\alpha^2=\sup_{t \leq T}|e^{A_1 t}(-A_1)^\alpha u|^2 \leq C |u|_\alpha^2.$$
 Moreover by \eqref{eq:bound_-A_exp At} and by Lemma \ref{lemma:a_priori_estimates} we have:
\begin{align*}
  \mathbb{E}\left [ \sup_{t\in [0,T]}  \left | \int_0^t e^{A_1(t-s)} F(U_s^{\epsilon},V_s^{\epsilon}) ds \right|_\alpha^2  \right] & \leq  \mathbb{E}\left [ \sup_{t\in [0,T]}  \left | \int_0^t (t-s)^{-\alpha} \left | F(U_s^{\epsilon},V_s^{\epsilon}) \right |  ds \right|^2  \right]  \\
&   \leq C    \mathbb{E}\left [ \sup_{t\in [0,T]}  \int_0^t  (t-s)^{-2\alpha} ds \int_0^t \left  | F(U_s^{\epsilon},V_s^{\epsilon}) \right |^2 ds \right] \\
& \leq C \int_0^T \left(1+ \mathbb{E}| U_s^{\epsilon}|^2 +\mathbb{E} |V_s^{\epsilon} |^2 \right) ds \\
& \leq C (1+|u|^2+|v|^2).
\end{align*}
Finally by Lemma \ref{lemma:estimate_stoch_convolution} 
we have:
\begin{equation*}
\mathbb{E} \left[  \sup_{t\in [0,T]} \abs{ \Gamma^1_t}^2_\alpha \right]  < +\infty.
\end{equation*}
By considering \eqref{eq:mild_U}, calculating  $ \mathbb{E}\left [ \sup_{t\in [0,T]} \abs{U^\epsilon_t}_{\alpha}^2 \right] $ and using the last three inequalities we have the thesis.
 \end{proof}
 \begin{lemma}\label{lemma:continuity_from_time_U}
 Let $0 < \alpha < 1/2 - \gamma$, $u \in D((-A_1)^\alpha)$, $v \in K$, then there exists $C=C(T,\alpha)>0$ such that
 \begin{equation*}
 \mathbb{E}\left | U_{t+h}^\epsilon - U_{t}^\epsilon \right|^2 \leq C |h|^{2 \alpha}  (1+|u|_\alpha^2+|v|^2),
 \end{equation*}
 for every $\epsilon>0$, $0 \leq t \leq T$, $h\geq 0$ such that $t+h \leq T$.
 \end{lemma}
\begin{proof}
For $0 \leq t \leq T$, $h\geq 0$ such that $t+h \leq T$ we have
\begin{equation}
U^\epsilon_{t+h}-U^\epsilon_{t}=(e^{A_1 h}-I) U^\epsilon_{t} + \int_t^{t+h} e^{A_1(t+h-s)}F(U^\epsilon_{s},V^\epsilon_{s}) ds + \int_t^{t+h} e^{A_1(t+h-s)} dW^{Q_1}_s.
\end{equation}
Consider the first term on the right-hand-side, as $u \in D((-A_1)^\alpha)$ by \eqref{eq:holder_continuity_semigroup} and by Lemma \ref{lemma:estimate_U_norm_alpha} we have:
\begin{align*}
\mathbb{E} \left |(e^{A_1 h}-I) U^\epsilon_{t} \right|^2 & \leq C |h|^{2 \alpha} \mathbb{E} \left |U^\epsilon_{t} \right|^2_\alpha \\
& \leq C |h|^{2 \alpha}  (1+|u|_\alpha^2+|v|^2).
\end{align*} 
Consider now the second term on the right-hand-side, then by Lemma \ref{lemma:a_priori_estimates} we have:
\begin{align*}
\mathbb{E} \left |\int_t^{t+h} e^{A_1(t+h-s)}F(U^\epsilon_{s},V^\epsilon_{s}) ds   \right |^2 & \leq 
C |h| \int_t^{t+h} \left (1+\mathbb{E}  |U^\epsilon_{s}|^2+\mathbb{E} |V^\epsilon_{s}|^2 \right ) ds \\
& \leq C |h|^2  (1+|u|^2+|v|^2).
\end{align*} 
Finally for the third term on the right-hand-side by Ito's isometry and Hypothesis \ref{hp:hilbert_schmidt} we have:
\begin{align*}
\mathbb{E} \left |  \int_t^{t+h} e^{A_1(t+h-s)} dW_s^{Q_1} \right |^2 & =\int_t^{t+h} \norm{ e^{A_1(t+h-s)} Q_1^{1/2}}_{\mathcal{L}_2}^2 ds\\
& \leq C \int_t^{t+h} |t+h-s|^{-2\gamma} ds \\
& = C |h|^{1-2 \gamma}.
\end{align*} 
As by assumption $2 \alpha \wedge 2 \wedge  (1-2 \gamma)= 2\alpha$ then we have the thesis.
\end{proof}
\section{Fast motion}\label{sec:fast_motion}
In this section we study some classical properties of the fast motion. Consider 
\begin{equation}\label{eq:fast_motion_fixed_u}
\begin{cases}
dv^v_s=A_2v^v_sds+ G(v^v_s)ds+d w^{Q_2}_s, \\
v^v_0=v \in K,
\end{cases}
\end{equation}
for every $s \geq 0$ and for some  $Q_2$-Wiener process $w_s^{Q_2}$.\\
First define the semigroup related to the  fast motion by
\begin{equation}
P_s \phi(v) =\mathbb{E} \left[ \phi(v_s^{v})\right],
\end{equation}
for every $\phi \in B_B(K)$, $s \geq 0$.\\
Next recall that $\delta$ was defined by \eqref{eq:definition_delta}, then we have:
\begin{lemma}\label{lemma:continuity_fast_motion_fast_comp}
\begin{equation*}
\mathbb{E}\left[ \left|v_s^{v_1}-v_s^{v_2} \right|^2\right] \leq e^{-4\delta s} |v_1-v_2|^2,
\end{equation*}
for every $s \geq 0$, $v_1,v_2 \in K$. 
\end{lemma}
\begin{proof}
Define $\rho_s=v_s^{v_1}-v_s^{v_2}$, then by \eqref{eq:strong_dissipativity_G} we have:
\begin{align*}
\frac{d}{ds} |\rho_s|^2 & = 2 \langle A_2 \rho_s,\rho_s \rangle +2\langle G(v_s^{v_1})-G(v_s^{v_2}),\rho_s \rangle \\
& \leq - 4 \delta |\rho_s|^2 ,
\end{align*}
for every $s \geq 0$, $v_1,v_2 \in K$.\\
Then by taking the expectation and applying the comparison Theorem we have the thesis.
\end{proof}
Next we can show:
\begin{lemma}\label{lemma:moments_fast_motion}
Let $p\geq 1$, then there exists $C=C(p)>0$ such that:
\begin{equation*}
\mathbb{E}\left[ \left|v_s^{v} \right|^p\right] \leq C  (1+e^{-\delta p s} |v|^p),
\end{equation*}
for every $s \geq 0$, $v \in K$. 
\end{lemma}
\begin{proof}
Define $$\tilde \Gamma_s^{Q_2}=\int_0^s e^{A_2(s-r)}d w_r^{Q_2}.$$
First by Burkholder's inequality and Hypotheses \ref{hp:operator2} and \ref{hp:hilbert_schmidt} similarly to what is done for Lemma \ref{lemma:estimate_stoch_convolution_2} we have:
\begin{align}\label{eq:est_stoch_conv_fast_motion_slowed}
\mathbb{E}\left [ \abs{\tilde \Gamma^{Q_2}_s}^p \right] 
&\leq C \left( \int_0^{s} \norm{e^{A_2 (s-r)} Q_2^{1/2}}_{\mathcal{L}_2}^2 dr \right)^{p/2} \nonumber \\
& \leq C \left( \int_0^{s}  e^{-\lambda (s-r)}  \left(\frac{s-r}{2} \wedge 1 \right )^{-2 \gamma} d r \right)^{p/2} \leq C ,
\end{align}
for every $s \geq 0$.\\
Now set $\rho_s=v^{v}_s-\tilde \Gamma_s^{Q_2}$. Then for $p\geq 2$ we have:
\begin{align*}
\frac{1}{p}\frac{d}{ds} |\rho_s|^p & =  \langle A_2 \rho_s,\rho_s \rangle |\rho_s|^{p-2} +\langle G(\rho_s+\tilde \Gamma_s^{Q_2})-G(\tilde \Gamma_s^{Q_2}),\rho_s \rangle |\rho_s|^{p-2}  + \langle G(\tilde \Gamma_s^{Q_2}),\rho_s \rangle |\rho_s|^{p-2}   \\
& \leq - 2 \delta |\rho_s|^p+ C (1+|\tilde \Gamma_s^{Q_2}|)|\rho_s|^{p-1} \\
&\leq -  \delta |\rho_s|^p+ C (1+|\tilde \Gamma_s^{Q_2}|^p) .
\end{align*}
Then by the comparison Theorem and \eqref{eq:est_stoch_conv_fast_motion_slowed} we have:
\begin{align*}
\mathbb{E}\left|v_s^{u,v} \right|^p&  \leq C \mathbb{E} |\rho_s|^p + C \mathbb{E} |\tilde \Gamma_s^{Q_2}|^p \\
& \leq C \left [e^{-\delta p s}|v|^p + \int_0^s e^{-\delta p (s-r)} \left ( 1+\mathbb{E} |\tilde \Gamma_r^{Q_2}|^p \right)  dr + \mathbb{E} |\tilde \Gamma_s^{Q_2}|^p  \right ]\\
& \leq C \left (e^{-\delta p s}|v|^p +1\right).
\end{align*}
\end{proof}
Now by \cite[Theorem 6.3.3]{daprato_ergodicity} there exists a unique invariant measure $\mu$ for the semigroup $P_t$. Moreover we have:
\begin{lemma}\label{lemma:moments_invariant_measure}
We have:
\begin{equation}
\int_K |z| \mu (dz) < \infty.
\end{equation}
\end{lemma}
\begin{proof}
Fix $N>0$, then by definition of invariant measure and Lemma \ref{lemma:moments_fast_motion} we have for every $s>0$
\begin{align*}
\int_K (|z|\wedge N) \mu (dz)  & = \int_K  \mathbb{E} \left [ |v_s^{z}| \wedge N \right]   \mu (dz) \leq \int_K \left \{  \left  [ C \left ( 1 +  e^{-\delta s}  |z|\right) \right ] \wedge N \right \} \mu (dz)\\
&  \leq \int_K  C \left ( 1\wedge N +  e^{-\delta s}  |z|\wedge N \right)  \mu (dz)  \leq C \left ( 1 +  e^{-\delta s} \int_K  (|z|\wedge N)  \mu (dz) \right),
\end{align*}
 where we have used the fact that $(a+b)\wedge c \leq a \wedge c + b \wedge c$ for every $a,b,c \geq 0.$
By choosing $s>0$ large enough we have 
\begin{align*}
\int_K (|z|\wedge N) \mu (dz)  \leq C,
\end{align*}
for some $C>0$ independent of $N$. Letting $N \to \infty$ by the monotone convergence Theorem we have the result.
\end{proof}
Next we study the convergence to equlibrium of the semigroup of the fast motion, i.e. we prove:
\begin{lemma}\label{lemma:convergence_equilibrium}
There exists $C>0$ such that
\begin{align*}
\left |P_s \phi(v)-\int_K \phi(z) \mu(dz) \right|  \leq C \left[ \phi\right]_{Lip}e^{-2\delta s} (1+|v|),
\end{align*}
for every $s \geq 0$, $v \in K$, $\phi \in Lip(K)$.\\
Moreover there exists $C>0$ such that
\begin{align*}
\left |P_s \phi(v)-\int_K \phi(z) \mu(dz) \right|  \leq C |\phi|_{\infty} e^{-\delta s} (s \wedge 1)^{-1/2}  (1+|v|) ,
\end{align*}
for every $s > 0$, $\phi \in B_B(K)$, $v \in K$.
\end{lemma}
\begin{proof}
First for every $\phi \in Lip(K)$ by Lemma \ref{lemma:continuity_fast_motion_fast_comp} we have:
\begin{equation}\label{eq:continuity_semigroup_fast_comp}
\begin{aligned}
\left|P_{s} \phi(v_2)-P_{s} \phi(v_1)\right| & \leq[\phi]_{\text {Lip }} \mathbb{E}\left[ \left|v^{v_2}_s-v^{ v_1}_s\right| \right] \\
& \leq[\phi]_{\text {Lip }} e^{-2\delta s}|v_2-v_1|,
\end{aligned}
\end{equation}
for every $s\geq0$, $v_1,v_2 \in K$.\\
Now let $s>0$, by definition of invariant measure, \eqref{eq:continuity_semigroup_fast_comp} and Lemma  \ref{lemma:moments_invariant_measure} we have:
\begin{align}\label{eq:proof_convergence_equilibrium}
\left |P_s \phi(v)-\int_K \phi(z) \mu(dz) \right| & = \left |\int_K \left ( P_{s} \phi(v)-P_{s} \phi(z) \right) \mu(dz)\right | \nonumber \\
& \leq  \left[ \phi\right]_{Lip} e^{-2\delta s} \int_{K} |v-z| \mu(dz) \nonumber\\
& \leq C \left[ \phi\right]_{Lip} e^{-2\delta s} (1+|v|) ,
\end{align}
for every $v \in K$, $\phi \in Lip(K)$ so that we have the first inequality.\\
Now thanks to Hypothesis \ref{hp:covariance_inverse} we can apply \cite[Theorem 9.32]{daprato_stochastic_eq} to have the Bismut-Elworthy formula:
\begin{equation}\label{eq:est_gradient_semigroup}
\sup_{u \in H} |D P_s \phi |_{\infty} \leq C (s \wedge 1)^{-1/2}  |\phi |_{\infty},
\end{equation}
for every $s>0$, $\phi \in B_B(K)$.\\
Now by the semigroup property, the regularizing property of the semigroup \eqref{eq:est_gradient_semigroup} and by \eqref{eq:continuity_semigroup_fast_comp}  we have:
\begin{align}\label{eq:continuity_semigroup_fast_comp_phi_bounded}
\left|P_{s} \phi(v_2)-P_{s} \phi(v_1)\right| &= \left|P_{s/2} \left (P_{s/2} \phi \right)  (v_2) -P_{s/2} \left (P_{s/2} \phi \right)(v_1) \right| \nonumber \\
&  \leq\left[P_{s / 2} \phi\right]_{\text{Lip}} e^{-2\delta \frac{s}{2}}|v_2-v_1| \nonumber \\
 & \leq C |\phi|_{\infty}(s \wedge 1)^{-\frac{1}{2}} e^{-\delta s}|v_2-v_1|,
\end{align}
for every $s>0$, $v_1,v_2 \in K$.\\
Finally similarly to before by \eqref{eq:continuity_semigroup_fast_comp_phi_bounded} for $s>0$ we have:
\begin{align*}\label{eq:proof_convergence_equilibrium_phi_bounded}
\left |P_s \phi(v)-\int_K \phi(z) \mu(dz) \right| & = \left |\int_K \left ( P_{s} \phi(v)-P_{s} \phi(z) \right) \mu(dz)\right | \\
& \leq C |\phi|_{\infty}(s \wedge 1)^{-\frac{1}{2}} e^{-\delta s} \int_{K} |v-z| \mu(dz) \\
& \leq C|\phi|_{\infty}(s \wedge 1)^{-\frac{1}{2}} e^{-\delta s} (1+|v|) ,
\end{align*}
for every $v \in K$, $\phi \in B_B(K)$.
\end{proof}
Now we study the mixing properties of the semigroup of the fast motion. To this purpose define for $0 \leq s \leq t \leq \infty$, $v \in K$
$$\mathcal{H}_s^{t}(v)= \sigma (v_r^{v}, 0 \leq s \leq r \leq t ).$$
Then a classical consequence of Lemmas \ref{lemma:moments_fast_motion}, \ref{lemma:convergence_equilibrium} is the following mixing result whose proof is the same of \cite[Lemma 3.2]{Cerrai_normal_dev} and is reported in the appendix for completeness.
\begin{lemma}\label{lemma:mixing_1}
There exists $C>0$ such that
\begin{equation*}
\sup \left \{  \left |\mathbb{P} (B_1 \cap B_2)-\mathbb{P} (B_1) \mathbb{P} (B_2) \right | : B_1 \in \mathcal{H}_0^{t}(v), B_2 \in \mathcal{H}_{t+s}^{\infty}(v) \right \} \leq C e^{-\delta s} s^{-1/2} (1+|v|),
\end{equation*}
for every $0 \leq s \leq t$, $v \in K$.
\end{lemma}
Now Lemma \ref{lemma:mixing_1} implies the following classical result, e.g. see \cite{rozanov} (see also \cite[Proposition 3.3]{Cerrai_normal_dev}). The proof can be found in the appendix for completeness.
\begin{lemma}\label{lemma:mixing_2}
There exists $C>0$ such that for every $0\leq s_1 \leq t_1 < s_2 \leq t_2$ and $\xi_i$ $\mathcal{H}_{s_i}^{t_i}(v)$-measurable $i=1,2$ and $|\xi_i|\leq 1$ a.s $i=1,2$ 
\begin{equation*}
\left|\mathbb{E}\left[  \xi_{1} \xi_{2}\right]- \mathbb{E} \xi_{1} \mathbb{E} \xi_{2}\right| \leq C \frac{e^{-\delta\left(s_{2}-t_{1}\right)}}{\sqrt{s_{2}-t_{1}}}\left(1+|v|\right).
\end{equation*}
\end{lemma}
Since in our case $|\xi_i|$ will not be bounded by $1$ we need the following result which is similar  to \cite[Proposition 3.3]{Cerrai_normal_dev}. Also in this case we postpone the proof in the appendix.
\begin{lemma}\label{lemma:mixing_3}
Let $\rho \in (0,1)$. Then there exists $C=C(\rho)>0$ such that for every $0\leq s_1 \leq t_1 < s_2 \leq t_2$ and $\xi_i$ $\mathcal{H}_{s_i}^{t_i}(v)$-measurable, $i=1,2$ satisfying for some $K_i=K_i(\rho)>0$
\begin{equation}\label{eq:lemma-hp_mixing_3}
\left( \mathbb{E}\left|\xi_{i}\right |^{\frac{2}{1-\rho}} \right)^{\frac{1-\rho}{2}}= K_i < \infty
\end{equation}
then:
$$
\left|\mathbb{E}\left[  \xi_{1} \xi_{2}\right]- \mathbb{E} \xi_{1} \mathbb{E} \xi_{2}\right| \leq C K_1^{\frac{2}{2+\rho}} K_2^{\frac{2}{2+\rho}} (K_1+K_2)^{\frac{2 \rho}{2+\rho}}  \left(\frac{e^{-\delta ( s_{2}-t_{1})}}{\sqrt{s_{2}-t_{1}}} \left(1+|v|\right) \right)^{\frac{\rho}{2+\rho}} .
$$
\end{lemma}
\section{Averaged equation}\label{sec:averaged_eq}
In this section we introduce the averaged equation and we prove its well-posedness.\\
\begin{equation}\label{eq:f_bar}
\overline{F}(u)=\int_K F(u,v)\mu(dv),
\end{equation}
for every $u\in H$ and consider the so called averaged equation:
\begin{equation}\label{eq:averaged_eq}
\begin{cases}
dU_t=A_1U_t dt+\overline{F}(U_t)dt+dW^{Q_1}_t ,\\
U_0=u  ,
\end{cases}
\end{equation}
for every $t\leq T$.\\
Now we prove the well-posedness of the averaged equation:
\begin{proposition}
Equation \eqref{eq:averaged_eq} admits a unique mild solution given by:
\begin{equation*}
U_t=e^{A_1t}u +\int_0^te^{A_1(t-s)} \overline{F} (U_s)ds +\int_0^te^{A_1(t-s)} dW^{Q_1}_s,
\end{equation*}
for every $t \in [0,T]$.\\
Moreover for every $p>0$ there exists $C=C(T,p)>0$ such that
\begin{align*} \mathbb{E}\left [ \sup_{t\in [0,T]} \abs{U_t}^p \right] \leq C (1+|u|^p).
\end{align*}
\end{proposition}
\begin{proof}
In order to prove the first part of the Proposition it is sufficient to prove that $\overline F$ is Lipschitz (e.g. see \cite{daprato_stochastic_eq}). But this follows from the Lipschitz continuity of $F$, indeed:
$$|\overline F (u_1) - \overline F (u_2)| \leq \int_K | F(u_1,v)-F(u_2,v) |\mu(dv) \leq L_F |u_1-u_2|, \quad \forall u_1,u_2 \in H.$$
Hence we obtain the Lipschitzianity of $\overline F$ and the first claim of the Proposition.\\
The proof of the second claim of the Proposition follows by a standard application of Gronwall's  Lemma.
\end{proof}
\section{Preliminary results}\label{sec:preliminary_results}
In this section we prove a technical result, i.e. Lemma \ref{lemma:crucial_lemma}, which is inspired by \cite[Lemma 4.2]{Cerrai_normal_dev} and follows by the mixing properties of the fast motion studied in section \ref{sec:fast_motion}, in particular Lemma \ref{lemma:mixing_3}. In order to prove it we proceed with similar techniques to the  ones of the proof of \cite[Lemma 4.2]{Cerrai_normal_dev}. However note that Lemma \ref{lemma:crucial_lemma} is a slightly stronger result than \cite[Lemma 4.2]{Cerrai_normal_dev}: indeed in Lemma \ref{lemma:crucial_lemma} the absolute value is inside the integral, while in \cite[Lemma 4.2]{Cerrai_normal_dev} it is outside.

Fix $\overline \xi \in \mathcal C([0,T];H)$, $v \in K, h\in H$ with $|h|=1$, and define 
$$\Phi_h(r) \coloneqq    \langle F( \overline \xi_{r},V_r^\epsilon)-\overline{F}(\overline \xi_r),h\rangle \quad \forall r \leq T$$
and
\begin{align*}
\Psi_h(r)  \coloneqq \Phi_h(\epsilon r)  & =  \langle F(\overline\xi_{\epsilon r},V_{\epsilon r}^\epsilon)-\overline{F}(\overline \xi_{\epsilon r}),h\rangle.
\end{align*}
Moreover let $n \in \mathbb{N}$ and set:
\begin{equation}\label{eq:definition_J_j}
\begin{aligned}
J_{ j}\left(r_{1}, \ldots, r_{2 n}\right) \coloneq \left|\mathbb{E} \prod_{i=1}^{2 n} \Psi_h\left(r_{i}\right)-\mathbb{E} \prod_{i=1}^{2 j} \Psi_h\left(r_{i}\right) \mathbb{E} \prod_{i=2 j+1}^{2 n} \Psi_h\left(r_{i}\right)\right|,
\end{aligned}
\end{equation}
for every $1\leq j\leq n$, $0 \leq r_1 \leq ... \leq r_{2n} \leq T/\epsilon$.\\
First we show the following result:
\begin{lemma}
Let $0< \rho < 1$ then there exists $C=C(T,\rho,n)>0$ and $\eta=\eta(\rho,j_2-j_1)>0$ such that 
\begin{equation}\label{eq:lemma_estimate_product}
\left | \mathbb{E}  \prod_{i=j_1}^{j_2} \Psi_h (r_i) \right| \leq C \left( \frac{ e^{-\delta (r_{j_2}-r_{j_2-1})}} {\sqrt{r_{j_2}-r_{j_2-1}}}\right)^{\frac{\rho}{2+\rho}} \left(1+\sup_{r \leq T}|\overline \xi_r|^{\eta}+|v|^{ \eta}\right),
\end{equation}
for every  $u \in H,v \in K$, $1\leq j_1 \leq j_2 \leq n$, $0 \leq r_1 \leq ... \leq r_{2n} \leq T/\epsilon$.\\
Moreover there exists $C=C(T,\rho,n)>0$ and $\eta= \eta(\rho,n) >0$ such that 
\begin{equation}\label{eq:lemma_estimate_mixing}
\begin{aligned}
J_{ j}\left(r_{1}, \ldots, r_{2 n}\right) \leq C \left(\frac{e^{-\delta \hat{r}_{j}}}{\sqrt{\hat{r}_{j}}}\right)^{\frac{\rho}{2+\rho}} \left(1+ \sup_{r \leq T}|\overline \xi_r| ^{\eta}+|v|^{\eta}\right) ,
\end{aligned}
\end{equation}
where
$$
\hat{r}_{j}=\max \left(r_{2 n}-r_{2 n-1}, r_{2 j+1}-r_{2 j}\right),
$$
for every $1\leq j\leq n$, $0 \leq r_1 \leq ... \leq r_{2n} \leq T/\epsilon$.
\end{lemma}
\begin{remark}
Note that the dependence of the exponents with respect to $j_2-j_1$ and $n$ is not restrictive: once $n$  has been fixed (together with $\rho$ and $T$) we are allowed to take any $0 \leq j_1 \leq j_2 \leq n$. In this sense $\eta=\eta(\rho, j_2-j_1)$. Of course since $j_2-j_1 \leq n$ we could choose $\eta'=\eta'(\rho,n) \geq \eta $ and replace $\eta$ with $\eta'$ in the estimates of the Lemma. However the estimate with  $\eta$ is more precise.
\end{remark}
\begin{proof}
By the sublinearity of $\Psi_h$ and Lemma \ref{lemma:moments_fast_motion} for every $p\geq1$ we have:
\begin{equation}\label{eq:proof_estimate_product}
\begin{aligned}
\mathbb{E}\left|\prod_{i=j_{1}}^{j_{2}} \Psi_h\left(r_{i}\right)\right|^{p} & \leq C \left(1+\sum_{i=j_{1}}^{j_{2}}  \left |\overline \xi_{\epsilon r_i} \right|^{\left(j_{2}-j_{1}+1\right) p}+\sum_{i=j_{1}}^{j_{2}} \mathbb{E}\left|V^\epsilon_{\epsilon r_{i}}\right|^{\left(j_{2}-j_{1}+1\right) p}\right) \\
& \leq C\left(1+\sup_{r \leq T}|\overline \xi_r|^{\left(j_{2}-j_{1}+1\right) p}+|v|^{\left(j_{2}-j_{1}+1\right) p}\right),
\end{aligned}
\end{equation}
for every $1 \leq j_1 \leq j_2 \leq 2n$.\\
Notice that  $V^{\epsilon}_{\epsilon r}$ and $v^v_r$, defined by \eqref{eq:fast_motion_fixed_u} for $w_r^{Q_2}= W^{Q_2}_{\epsilon r}/ \sqrt{\epsilon}$, are indistinguishable for $r \in [0,T/\epsilon]$. Then by setting $p=2/(1-\rho)$ and applying Lemma  \ref{lemma:mixing_3} to $\xi_1=\prod_{i=1}^{j} \Psi_h \left(r_{i}\right)$, $\xi_2=\prod_{i=j+1}^{2 n} \Psi_h \left(r_{i}\right)$ with $K_1=C(1+\sup_{r \leq T}|\overline \xi_r|^j+|v|^j)$, $K_2=C(1+\sup_{r \leq T}|\overline \xi_r|^{2n-j}+|v|^{2n-j})$  for every $0 \leq j_1 <  j < j_2 \leq 2n$ we have:
\begin{equation}\label{eq:proof_estimate_mixing}
\begin{aligned}
&\left|\mathbb{E} \prod_{i=j_1}^{j_2} \Psi_h\left(r_{i}\right)-\mathbb{E} \prod_{i=j_1}^{j} \Psi_h \left(r_{i}\right) \mathbb{E} \prod_{i=j+1}^{j_2} \Psi_h \left(r_{i}\right)\right|  \leq C \left(1+\sup_{r \leq T}|\overline \xi_r|^{\eta}+|v|^{\eta}\right)\left(\frac{e^{-\delta\left(r_{j+1}-r_{j}\right)}}{\sqrt{r_{j+1}-r_{j}}}\right)^{\frac{\rho}{2+\rho}},
\end{aligned}
\end{equation}
where $\eta=j_2-j_1 + \frac{\rho (j_2-j_1-1)}{(\rho+2)}$. \\
By definitions of $\Psi_h$, the indistinguishability of $V^{\epsilon}_{\epsilon r}$, $v^v_r$ and by Lemma \ref{lemma:convergence_equilibrium} we have:
\begin{equation}\label{eq:estimate_phi}
\left|\mathbb{E} \Psi_h\left(r_{j_{2}}\right)\right| \leq C e^{-2 \delta r_{j_{2}}}\left(1+\sup_{r \leq T}|\overline \xi_r|+|v|\right).
\end{equation}
Now by the last three inequalities we have: 
\begin{small}
\begin{equation*}
\begin{aligned}
\left|\mathbb{E} \prod_{i=j_{1}}^{j_{2}} \Psi_h \left(r_{i}\right)\right| & \leq \left|\mathbb{E} \prod_{i=j_{1}}^{j_{2}} \Psi_h\left(r_{i}\right)-\mathbb{E} \prod_{i=j_{1}}^{j_{2}-1} \Psi_h \left(r_{i}\right) \mathbb{E} \Psi_h \left(r_{j_{2}}\right)\right|+ \mathbb{E} \Bigg| \prod_{i=j_{1}}^{j_{2}-1} \Psi_h \left(r_{i}\right) \Bigg| \big|    \mathbb{E} \Psi_h \left(r_{j_{2}}\right)\big|   \\
& \leq C \left(1+\sup_{r \leq T}|\overline \xi_r|^{ \eta}+|v|^{ \eta}\right)\left(\frac{e^{-\delta\left(r_{j_{2}}-r_{j_{2}-1}\right)}}{\sqrt{r_{j_{2}}-r_{j_{2}-1}}}\right)^{\frac{\rho}{2+\rho}}\\
& \quad + C (1+\sup_{r \leq T}|\overline \xi_r|^{j_{2}-j_{1}}+|v|^{j_{2}-j_{1}}) \left(1+\sup_{r \leq T}|\overline \xi_r|+|v|\right)   e^{-2 \delta r_{j_{2}}} \\
& \leq C\left(1+\sup_{r \leq T}|\overline \xi_r|^{ \eta}+|v|^{ \eta}\right)\left(\frac{e^{-\delta\left(r_{j_{2}}-r_{j_{2}-1}\right)}}{\sqrt{r_{j_{2}}-r_{j_{2}-1}}}\right)^{\frac{\rho}{2+\rho}} ,
\end{aligned}
\end{equation*}
\end{small}
where $ \eta=(j_2-j_1 +1)+ \frac{\rho (j_2-j_1)}{(\rho+2)}$.
This implies \eqref{eq:lemma_estimate_product}.\\
Now by \eqref{eq:lemma_estimate_product} we have:
\begin{equation}
\left|\mathbb{E} \prod_{i=j_1}^{2 n} \Psi_h \left(r_{i}\right)\right| \leq C \left(1+\sup_{r \leq T}|\overline \xi_r|^{\eta}+|v|^{\eta}\right) \left(\frac{e^{-\delta\left(r_{2 n}-r_{2 n-1}\right)}}{\sqrt{r_{2 n}-r_{2 n-1}}}\right)^{\frac{\rho}{2+\rho}}.
\end{equation}
for every $j_1 <2n.$\\
Now fix any  $j < 2n$. Then by  the last inequality and \eqref{eq:proof_estimate_product} we have:
\begin{small}
\begin{equation}\label{eq:proof_estimate_mixing_products_final_variables}
\begin{aligned}
&\left|\mathbb{E} \prod_{i=1}^{2 n} \Psi_h \left(r_{i}\right)-\mathbb{E} \prod_{i=1}^{j} \Psi_h \left(r_{i}\right) \mathbb{E} \prod_{i=j+1}^{2 n} \Psi_h \left(r_{i}\right)\right|  \leq C\left(1+\sup_{r \leq T}|\overline \xi_r|^{\eta}+|v|^{\eta}\right)\left(\frac{e^{-\delta\left(r_{2 n}-r_{2 n-1}\right)}}{\sqrt{r_{2 n}-r_{2 n-1}}}\right)^{\frac{\rho}{2+\rho}}.
\end{aligned}
\end{equation}
\end{small}
Finally consider \eqref{eq:proof_estimate_mixing} with $j_1=1,j_2=2n$ and \eqref{eq:proof_estimate_mixing_products_final_variables}. Since the function $f(s)=e^{-\delta s} s^{-1/2}$ is decreasing we have
\begin{small}
\begin{equation*}
\begin{aligned}
\left|\mathbb{E} \prod_{i=1}^{2 n} \Psi_h \left(r_{i}\right)-\mathbb{E} \prod_{i=1}^{j} \Psi_h \left(r_{i}\right) \mathbb{E} \prod_{i=j+1}^{2 n} \Psi_h \left(r_{i}\right)\right|&\leq C \left(\frac{e^{-\delta \max \left(r_{2 n}-r_{2 n-1}, r_{2 j+1}-r_{2 j}\right)}}{\sqrt{\max \left(r_{2 n}-r_{2 n-1}, r_{2 j+1}-r_{2 j}\right)}}\right)^{\frac{\rho}{2+\rho}} \left(1+\sup_{r \leq T}|\overline \xi_r|^{\eta}+|v|^{\eta}\right)\\
&=C \left(\frac{e^{-\delta \hat{r}_{j}}}{\sqrt{\hat{r}_{j}}}\right)^{\frac{\rho}{2+\rho}} \left(1+ \sup_{r \leq T}|\overline \xi_r| ^{\eta}+|v|^{\eta}\right) .
\end{aligned}
\end{equation*}
\end{small}
This implies \eqref{eq:lemma_estimate_mixing}.
\end{proof}
Let $0<\alpha$, $0 \leq \beta < 1/3$ and set 
$$\theta_{\alpha,\beta}(r):=e^{-r \alpha}r^{-\beta} \quad \forall r>0.$$

Now we can state and prove the main result of this section.
\begin{lemma}\label{lemma:crucial_lemma}
Let $n \in \mathbb{N}$ and $0 \leq \beta < 1/3$. Then there exists a constant $C=C(T,n,\beta)>0$ and $\eta=\eta(n)>0$ such that for every $\epsilon>0$, $\alpha>0$, $u \in H,v \in K$, $h\in H$ with $|h|=1$ we have:
\begin{equation*}
 \int_{[s,t]^{2n}} \left  | \mathbb{E}  \prod_{i=1}^{2n} \theta_{\alpha,\beta}(t-r_i) \langle F(\overline \xi_{r_i},V_{r_i}^\epsilon)-\overline{F}(\overline \xi_{ r_i}),h\rangle\right|    dr_1...dr_{2n} \leq C \left(1+\sup_{r \leq T}|\overline \xi_r|^{\eta}+|v|^{\eta}\right) \epsilon^{n} \left( \frac{1}{\alpha} \right)^{(1-2\beta)n}.
\end{equation*}
\end{lemma}
\begin{proof}
Recall the definition of $\Psi_h(r)$, then by a change of variable we have:
\begin{equation}\label{eq:first_change_variable_product_integrals}
\begin{aligned}
 \int_{[s,t]^{2n}}  \left |  \mathbb{E}  \prod_{i=1}^{2n} \theta_{\alpha,\beta}(t-r_i) \langle F(\overline \xi_{r_i},V_{r_i}^\epsilon)-\overline{F}(\overline \xi_{ r_i}),h\rangle \right|   dr_1...dr_{2n} =\epsilon^{2 n}   H_{\epsilon}(s, t) ,
\end{aligned}
\end{equation}
where we have defined:
$$
H_{\epsilon}(s, t) :=\int_{\left[\frac{s}{\epsilon}, \frac{t}{\epsilon}\right]^{2 n}}  \left |  \mathbb{E}   \prod_{i=1}^{2 n} \left ( \theta_{\alpha, \beta}\left(t-\epsilon r_{i}\right) \Psi_{ h}\left(r_{i}\right) \right) \right|   d r_{1} \cdots d r_{2 n}.
$$
By simmetry we have:
\begin{equation}\label{eq:H_eps}
 H_{\epsilon}(s, t)=C \int_{\frac{s}{\epsilon}}^{\frac{t}{\epsilon}} \int_{\frac{s}{\epsilon}}^{r_{2 n}} ... \int_{\frac{s}{\epsilon}}^{r_{2 }}  \left | \mathbb{E}  \prod_{i=1}^{2n} \left ( \theta_{\alpha, \beta}\left(t-\epsilon r_{i}\right) \Psi_{h}\left(r_{i}\right) \right) \right|  d r_{1} \cdots d r_{2 n}.
\end{equation}
We proceed by induction on $n$ and to this purpose fix some $\rho \in (0,1)$. Consider $n=1$ then by the definition of $\theta_{\alpha,\beta}=e^{-r \alpha}r^{-\beta}$, \eqref{eq:lemma_estimate_product} and some changes of variables we have
\begin{equation*}
\begin{aligned}
\epsilon^2   H_{\epsilon}(s, t)&=2 \epsilon^2 \int_{\frac{s}{\epsilon}}^{\frac{t}{\epsilon}} \int_{\frac{s}{\epsilon}}^{r_{2}} \theta_{\alpha, \beta}\left(t-\epsilon r_{1}\right) \theta_{\alpha, \beta}\left(t-\epsilon r_{2}\right)   \left| \mathbb{E} \left [ \Psi_{h}\left(r_{1}\right) \Psi_{h}\left(r_{2}\right) \right] \right|  d r_{1} d r_{2} \\
& \leq C \epsilon^2 \int_{\frac{s}{\epsilon}}^{\frac{t}{\epsilon}} \int_{\frac{s}{\epsilon}}^{r_{2}} \theta_{\alpha, \beta}\left(t-\epsilon r_{1}\right) \theta_{\alpha, \beta}\left(t-\epsilon r_{2}\right)\left(\frac{e^{-\delta\left(r_{2}-r_{1}\right)}}{\sqrt{r_{2}-r_{1}}}\right)^{\frac{\rho}{2+\rho}} d r_{1} d r_{2} \left(1+\sup_{r \leq T}|\overline \xi_r|^{\eta}+|v|^{\eta}\right)\\
&=C \epsilon^{2-2 \beta} \int_{0}^{\frac{t-s}{\epsilon}} y_{2}^{-\beta} e^{-\epsilon y_{2} \alpha} \int_{y_{2}}^{\frac{t-s}{\epsilon}} y_{1}^{-\beta} e^{-\epsilon y_{1} \alpha}\left(\frac{e^{-\delta\left(y_{1}-y_{2}\right)}}{\sqrt{y_{1}-y_{2}}}\right)^{\frac{\rho}{2+\rho}} d y_{1} d y_{2} \left(1+\sup_{r \leq T}|\overline \xi_r|^{\eta}+|v|^{\eta}\right)\\
&\leq C \epsilon^{2-2 \beta} \int_{0}^{\frac{t-s}{\epsilon}} y_{2}^{-2 \beta} e^{-2 \epsilon y_{2} \alpha} \int_{0}^{+\infty}\left(\frac{e^{-\delta y_{1}}}{\sqrt{y_{1}}}\right)^{\frac{\rho}{2+\rho}} d y_{1} d y_{2} \left(1+\sup_{r \leq T}|\overline \xi_r|^{\eta}+|v|^{\eta}\right)\\
& \leq C \epsilon^{2-2 \beta} \int_{0}^{\frac{t-s}{\epsilon}} y_2^{-2 \beta} e^{-2 \epsilon y_2 \alpha} d y_2 \left(1+\sup_{r \leq T}|\overline \xi_r|^{\eta}+|v|^{\eta}\right)\\
& = C  \epsilon^{2-2 \beta}(\epsilon \alpha)^{-(1-2 \beta)} \int_{0}^{\alpha(t-s)} r^{-2 \beta} e^{-2 r} d r \left(1+\sup_{r \leq T}|\overline \xi_r|^{ \eta}+|v|^{\eta}\right) \\
& \leq C \epsilon^{2-2 \beta}(\epsilon \alpha)^{-(1-2 \beta)}\left(1+\sup_{r \leq T}|\overline \xi_r|^{\eta}+|v|^{\eta}\right)\\
&=C \epsilon \alpha^{-(1-2 \beta)}\left(1+\sup_{r \leq T}|\overline\xi_r|^{\eta}+|v|^{\eta}\right),
\end{aligned}
\end{equation*}
so that by \eqref{eq:first_change_variable_product_integrals} we have the thesis for $n=1$.

Now assume that 
\begin{equation*}
 \int_{[s,t]^{j}}   \left | \mathbb{E} \prod_{i=1}^{j} \theta_{\alpha,\beta}(t-\epsilon r_i) \Psi_h (r_i)\right|   dr_1...dr_j  \leq C \left(1+\sup_{r \leq T}|\overline\xi_r|^{\eta_j}+|v|^{\eta_j}\right) \epsilon^{j/2} \left( \frac{1}{\alpha} \right)^{(1-2\beta)j/2},
\end{equation*}
for every even $j<2n$ where $\eta_j=j + \frac{\rho (j-1)}{(\rho+2)}$.\\
We prove that then it holds for $j=2n$.\\
Set for $r=(r_1,...r_{2n}) \in (s,t)^{2n}$ with $s \leq r_1 \leq ... \leq r_{2n} \leq t$ the integer $j(r)$ such that
\begin{equation*}
\max _{j=1, \ldots, n-1}\left(r_{2 j+1}-r_{2 j}\right)=r_{2 j(r)+1}-r_{2 j(r)}.
\end{equation*}
  and consider $H_\epsilon(s,t)$ given by \eqref{eq:H_eps}.     Recalling the definition of $J_j(r_1,...r_{2n})$ given by \eqref{eq:definition_J_j}  we have: 
 \begin{small}
\begin{equation*}
\begin{aligned}
\epsilon^{2n} H_{\epsilon}(s, t) &=C  \epsilon^{2n} \int_{\frac{s}{\epsilon}}^{\frac{t}{\epsilon}} \int_{\frac{s}{\epsilon}}^{r_{2 n}} ... \int_{\frac{s}{\epsilon}}^{r_{2 }}   \prod_{i=1}^{2n}\theta_{\alpha, \beta}\left(t-\epsilon r_{i}\right)   \left |  \mathbb{E}   \prod_{i=1}^{2n} \Psi_{h}\left(r_{i}\right) \right |  d r_{1} \cdots d r_{2 n} \\
& \leq C \epsilon^{2n} \int_{\frac{s}{\epsilon}}^{\frac{t}{\epsilon}} \int_{\frac{s}{\epsilon}}^{r_{2 n}} ... \int_{\frac{s}{\epsilon}}^{r_{2}} \prod_{i=1}^{2n} \theta_{\alpha, \beta} \left(t-\epsilon r_{i}\right) J_{ j(r)}\left(r_{1}, \ldots, r_{2 n}\right) d r_{1} \cdots d r_{2 n}\\
&\quad +C \epsilon^{2n}  \int_{\frac{s}{\epsilon}}^{\frac{t}{\epsilon}} \int_{\frac{s}{\epsilon}}^{r_{2 n}} ... \int_{\frac{s}{\epsilon}}^{r_{2}} \prod_{i=1}^{2n} \theta_{\alpha, \beta} \left(t-\epsilon r_{i}\right)\left|\mathbb{E} \prod_{i=1}^{2 j(r)} \Psi_{ h}\left(r_{i}\right)\mathbb{E} \prod_{i=2j(r)+1}^{2n} \Psi_{ h}\left( r_{i}\right)\right| 
 d r_{1} \cdots d r_{2 n}.\end{aligned}
\end{equation*}
 \end{small}
Note that by definition of $j(r)$,  for every $s \leq r_1 \leq ... \leq r_{2n} \leq t$, we have
 \begin{small}
$$\left|\mathbb{E} \prod_{i=1}^{2 j(r)} \Psi_{ h}\left(r_{i}\right)\mathbb{E} \prod_{i=2j(r)+1}^{2n} \Psi_{ h}\left( r_{i}\right)\right|  \leq\sum_{j=1}^{n-1} \left|\mathbb{E} \prod_{i=1}^{2 j} \Psi_{ h}\left(r_{i}\right)\mathbb{E} \prod_{i=2j+1}^{2n} \Psi_{ h}\left( r_{i}\right) \right|  . $$
 \end{small}
It follows:
  \begin{small}
\begin{equation}\label{eq:proof_lemma_inequality_for_H}
\begin{aligned}
 & \epsilon^{2n}   H_{\epsilon}(s, t)  \\
&    \leq C \epsilon^{2n} \int_{\frac{s}{\epsilon}}^{\frac{t}{\epsilon}} \int_{\frac{s}{\epsilon}}^{r_{2 n}} ... \int_{\frac{s}{\epsilon}}^{r_{2}} \prod_{i=1}^{2n} \theta_{\alpha, \beta} \left(t-\epsilon r_{i}\right) J_{ j(r)}\left(r_{1}, \ldots, r_{2 n}\right) d r_{1} \cdots d r_{2 n}  \\
&\quad   +C \epsilon^{2n}  \sum_{j=1}^{n-1}  \int_{\frac{s}{\epsilon}}^{\frac{t}{\epsilon}} \int_{\frac{s}{\epsilon}}^{r_{2 n}} ... \int_{\frac{s}{\epsilon}}^{r_{2}} \prod_{i=1}^{2n} \theta_{\alpha, \beta} \left(t-\epsilon r_{i}\right)\left|\mathbb{E} \prod_{i=1}^{2 j} \Psi_{ h}\left(r_{i}\right)\mathbb{E} \prod_{i=2j+1}^{2n} \Psi_{ h}\left( r_{i}\right)\right| 
 d r_{1} \cdots d r_{2 n}  \\
 &
= C \epsilon^{2n} \int_{\frac{s}{\epsilon}}^{\frac{t}{\epsilon}} \int_{\frac{s}{\epsilon}}^{r_{2 n}} ... \int_{\frac{s}{\epsilon}}^{r_{2}} \prod_{i=1}^{2n} \theta_{\alpha, \beta} \left(t-\epsilon r_{i}\right) J_{ j(r)}\left(r_{1}, \ldots, r_{2 n}\right) d r_{1} \cdots d r_{2 n}\\
&\quad +C \epsilon^{2n} \sum_{j=1}^{n-1} \int_{\frac{s}{\epsilon}}^{\frac{t}{\epsilon}} \int_{\frac{s}{\epsilon}}^{r_{2 j}} \ldots \int_{\frac{s}{\epsilon}}^{r_{2}} \prod_{i=1}^{2 j} \theta_{\alpha, \beta}\left(t-\epsilon r_{i}\right)\left|\mathbb{E} \prod_{i=1}^{2 j} \Psi_{ h}\left(r_{i}\right)\right| d r_{1} \cdots r_{2 j}\\
&\quad \quad  \times \int_{\frac{s}{\epsilon}}^{\frac{t}{\epsilon}} \int_{\frac{s}{\epsilon}}^{r_{2(n-j)}} \cdots \int_{\frac{s}{\epsilon}}^{r_{2}} \prod_{i=1}^{2(n-j)} \theta_{\alpha, \beta}\left(t-\epsilon r_{i}\right)\left|\mathbb{E} \prod_{i=1}^{2(n-j)} \Psi_{ h}\left( r_{i}\right)\right| d r_{1} \cdots r_{2(n-j)}\\
&=: \epsilon^{2n} I_{1, \epsilon} (s,t)+ \epsilon^{2n}I_{2, \epsilon} (s,t).
\end{aligned}
\end{equation}
 \end{small}
 
Now by \eqref{eq:lemma_estimate_mixing} and the definition of $j(r)$  we have:
\begin{align*}
J_{j(r)}\left(r_{1}, \ldots, r_{2 n}\right) &\leq C \left(1+\sup_{r \leq T}|\overline \xi_r|^{\eta}+|v|^{\eta}\right) \left(\frac{e^{-\delta \hat{r}_{j(r)}}}{\sqrt{\hat{r}_{j(r)}}}\right)^{\frac{\rho}{2+\rho}}\\
& =C \left(1+\sup_{r \leq T}|\overline \xi_r|^{\eta}+|v|^{\eta}\right) \left(\frac{e^{-\delta \max \left(r_{2 n}-r_{2 n-1}, r_{2n-1}-r_{2n-2},\cdots,r_{3}-r_{2}\right)}}{\sqrt{\max \left(r_{2 n}-r_{2 n-1}, r_{2n-1}-r_{2n-2},\cdots,r_{3}-r_{2}\right)}}\right)^{\frac{\rho}{2+\rho}}\\
&=C \left(1+\sup_{r \leq T}|\overline \xi_r|^{\eta}+|v|^{\eta}\right) \frac{e^{-\frac{\delta \rho}{2+\rho}  \max \left(r_{2 n}-r_{2 n-1}, r_{2n-1}-r_{2n-2},\cdots,r_{3}-r_{2}\right)}}{\left[\max \left(r_{2 n}-r_{2 n-1}, r_{2n-1}-r_{2n-2},\cdots,r_{3}-r_{2}\right) \right]^{\bar \rho}},
\end{align*}
 where $\bar{\rho}=\frac{\rho}{2(2+\rho)}$.\\
Recall that $r_{i+1} \geq r_{i}$ for every $i$ and note that $\max \left(r_{2 n}-r_{2 n-1}, r_{2n-1}-r_{2n-2}\right) \geq 1/2 (r_{2 n}-r_{2 n-2})$ and $\max \left(r_{2 n}-r_{2 n-1}, r_{2n-1}-r_{2n-2},\cdots,r_{3}-r_{2}\right) \geq \frac 1 n  \left(r_{2 n}-r_{2 n-1}+\sum_{i=1}^{n-1} r_{2 i+1}-r_{2 i}\right).$
Hence, since $g(t)=1/t^{\bar \rho}$, $f(t)=e^{-\alpha t}$ for $\alpha>0$ are decreasing, we have (recall that $C$ may change from line to line):
\begin{align*}
 J_{j(r)}\left(r_{1}, \ldots, r_{2 n}\right) 
&  \leq C \left(1+\sup_{r \leq T}|\overline \xi_r|^{\eta}+|v|^{\eta}\right) \frac{e^{-\frac{\delta \rho}{2+\rho}  \max \left(r_{2 n}-r_{2 n-1}, r_{2n-1}-r_{2n-2},\cdots,r_{3}-r_{2}\right)}}{\left[\max \left(r_{2 n}-r_{2 n-1}, r_{2n-1}-r_{2n-2}\right) \right]^{\bar \rho}} \\
&  \leq C \left(1+\sup_{r \leq T}|\overline \xi_r|^{\eta}+|v|^{\eta}\right) \frac{e^{-\frac{\delta \rho}{2+\rho}  \max \left(r_{2 n}-r_{2 n-1}, r_{2n-1}-r_{2n-2},\cdots,r_{3}-r_{2}\right)}}{\left(r_{2 n}-r_{2 n-2}\right)^{\bar{\rho}}}   \\
& \leq C \left(1+\sup_{r \leq T}|\overline \xi_r|^{\eta}+|v|^{\eta}\right)\frac{e^{-\frac{\delta \rho}{2+\rho} \frac 1 n  \left(r_{2 n}-r_{2 n-1}+\sum_{i=1}^{n-1} r_{2 i+1}-r_{2 i}\right)}}{\left(r_{2 n}-r_{2 n-2}\right)^{\bar{\rho}}}\\
& = C \left(1+\sup_{r \leq T}|\overline\xi_r|^{\eta}+|v|^{\eta}\right) \frac{e^{-\delta_{n}\left(r_{2 n}-r_{2 n-2}+\sum_{i=1}^{n-2} r_{2 i+1}-r_{2 i}\right)}}{\left(r_{2 n}-r_{2 n-2}\right)^{\bar{\rho}}},
\end{align*}
where in the last equality we have defined $\delta_{n}=\frac{\delta \rho}{n(2+\rho)}$ and we have used the following identity: $\sum_{i=1}^{n-1} r_{2 i+1}-r_{2 i}=r_{2n-1}-r_{2n-2}+ \sum_{i=1}^{n-2} r_{2 i+1}-r_{2 i}$.\\
We can apply this last inequality in order to estimate $I_{1,\epsilon}(s,t)$, i.e.
\begin{align*}
I_{1, \epsilon}(s,t)  & \leq C \left(1+\sup_{r \leq T}|\xi_r|^{\eta}+|v|^{\eta}\right)
\int_{\frac{s}{\epsilon}}^{\frac{t}{\epsilon}} \int_{\frac{s}{\epsilon}}^{r_{2 n}} \cdots  \int_{\frac{s}{\epsilon}}^{r_{2}} \frac{e^{-\delta_{n}\left(r_{2 n}-r_{2 n-2}\right)}}{\left(r_{2 n}-r_{2 n-2}\right)^{\bar{\rho}}} \\ &\times \prod_{i=1}^{2 n} \vartheta_{\alpha, \beta}\left(t-\epsilon r_{i}\right) \prod_{i=1}^{n-2} e^{-\delta_{n}\left(r_{2 i+1}-r_{2 i}\right)} d r_{1} \cdots d r_{2 n}\\
&=C \epsilon^{-2 n \beta} \int_{0}^{\frac{t-s}{\epsilon}} e^{-\left(\epsilon \alpha-\delta_{n}\right) y_{2 n}} y_{2 n}^{-\beta} \int_{y_{2 n}}^{\frac{t-s}{\epsilon}} e^{-\epsilon \alpha y_{2 n-1}} y_{2 n-1}^{-\beta} \int_{y_{2 n-1}}^{\frac{t-s}{\epsilon}} \frac{e^{-\left(\epsilon \alpha+\delta_{n}\right) y_{2 n-2}}}{\left(y_{2 n-2}-y_{2 n}\right)^{\bar{\rho}}} y_{2 n-2}^{-\beta}\\
&\times \int_{y_{2 n-2}}^{\frac{t-s}{\epsilon}} \cdots \int_{y_{3}}^{\frac{t-s}{\epsilon}} \prod_{i=2}^{2 n-2} e^{-\left(\epsilon \alpha+(-1)^{i} \delta_{n}\right) y_{i}} r_{i}^{-\beta} \int_{y_{2}}^{\frac{t-s}{\epsilon}} e^{-\epsilon \alpha y_{1}} y_{1}^{-\beta} d y_{1} \cdots d y_{2 n}  \left(1+\sup_{r \leq T}|\xi_r|^{\eta}+|v|^{\eta}\right).
\end{align*}
Now for $k=1,2,3$, $i=1,3, ... 2n-1$ we obtain:
\begin{align}\label{eq:proof_lemma_eq_odd}
\int_{y_{i+1}}^{\frac{t-s}{\epsilon}} e^{-k \epsilon \alpha y_{i}} y_{i}^{-k \beta} d y_{i}&=(\epsilon \alpha)^{k \beta-1} \int_{\epsilon \alpha y_{i+1}}^{(t-s) \alpha} e^{-k y_{i}} y_{i}^{-k \beta} d y_{i} \nonumber \\
& \leq  C (\epsilon \alpha)^{k \beta-1}.
\end{align}
Now consider $i=2,4,...2n$, we have:
\begin{align}\label{eq:proof_lemma_eq_even}
\int_{y_{i+1}}^{\frac{t-s}{\epsilon}} e^{-\left(\epsilon \alpha+\delta_{n}\right) y_{i}} y_{i}^{-\beta} d y_{i} & \leq y_{i+1}^{-\beta} e^{-\left(\epsilon \alpha+\delta_{n}\right) y_{i+1}}\left(\epsilon \alpha+\delta_{n}\right)^{-1} \nonumber \\ 
& \leq C y_{i+1}^{-\beta} e^{-\left(\epsilon \alpha+\delta_{n}\right) y_{i+1}}.
\end{align}
Now by \eqref{eq:proof_lemma_eq_odd} and \eqref{eq:proof_lemma_eq_even} we have:
\begin{equation}\label{eq:proof_I1_first_part}
\begin{aligned}
 \int_{y_{2 n-2}}^{\frac{t-s}{\epsilon}} \cdots  \int_{y_{3}}^{\frac{t-s}{\epsilon}} &  \prod_{i=2}^{2 n-2} e^{-\left(\epsilon \alpha+(-1)^{i} \delta_{n}\right) y_{i}} y_{i}^{-\beta} \int_{y_{2}}^{\frac{t-s}{\epsilon}} e^{-\epsilon \alpha y_{1}} y_{1}^{-\beta} d y_{1} \cdots d y_{2 n-3} \\
{ } &\leq(\epsilon \alpha)^{\beta-1} (\epsilon \alpha)^{(2 \beta-1)(n-2)} \\
& =(\epsilon \alpha)^{n(2 \beta-1)-(3 \beta-1)}.
\end{aligned}
\end{equation}
Now consider the remaining term in the inequality for $I_{1,\epsilon}(s,t)$: by \eqref{eq:proof_lemma_eq_odd} and \eqref{eq:proof_lemma_eq_even}  we have
\begin{small}
\begin{equation}\label{eq:proof_I1_second_part}
\begin{aligned}
\int_{0}^{\frac{t-s}{\epsilon}} e^{-\left(\epsilon \alpha-\delta_{n}\right) y_{2 n}} y_{2 n}^{-\beta} & \int_{y_{2 n}}^{\frac{t-s}{\epsilon}} e^{-\epsilon \alpha y_{2 n-1}} y_{2 n-1}^{-\beta} \int_{y_{2 n-1}}^{\frac{t-s}{\epsilon}} \frac{e^{-\left(\epsilon \alpha+\delta_{n}\right) y_{2 n-2}}}{\left(y_{2 n-2}-y_{2 n}\right)^{\bar{\rho}}} y_{2 n-2}^{-\beta} d y_{2 n-2} d y_{2 n-1} d y_{2 n} \\
&\leq C \int_{0}^{\frac{t-s}{\epsilon}} e^{-3 \epsilon \alpha y_{2 n}} y_{2 n}^{-3 \beta} \int_{y_{2 n}}^{\frac{t-s}{\epsilon}} \frac{e^{-\left(2 \epsilon \alpha+\delta_{n}\right)\left(y_{2 n-1}-y_{2 n}\right)}}{\left(y_{2 n-1}-y_{2 n}\right)^{\bar{\rho}}} d y_{2 n-1} d y_{2 n} \\
&\leq C \int_{0}^{\frac{t-s}{\epsilon}} e^{-3 \epsilon \alpha y_{2 n}} y_{2 n}^{-3 \beta} \int_{0}^{\infty} \frac{e^{-\delta_{n} y_{2 n-1}}}{y_{2 n-1}^{\bar{\rho}}} d y_{2 n-1} d y_{2 n} \\
&\leq C(\epsilon \alpha)^{3 \beta-1}.
\end{aligned}
\end{equation}
\end{small}
Applying now \eqref{eq:proof_I1_first_part} and \eqref{eq:proof_I1_second_part} to the inequality for $I_{1,\epsilon}(s,t)$ we have:
\begin{equation}
\begin{aligned}
I_{1, \epsilon} & \leq C \epsilon^{-2 n \beta} (\epsilon \alpha)^{n(2 \beta-1)-(3 \beta-1)} (\epsilon \alpha)^{3 \beta-1} \left(1+\sup_{r \leq T}|\overline \xi_r|^{\eta}+|v|^{\eta}\right) \\
&=C \epsilon^{-n} \alpha^{-(1-2 \beta)n} \left(1+\sup_{r \leq T}|\overline \xi_r|^{\eta}+|v|^{\eta}\right).
\end{aligned}
\end{equation}
In addition by the inductive hypothesis we have:
\begin{equation}
I_{2, \epsilon} \leq C  \epsilon^{-n} \alpha^{-(1-2 \beta)n} \left(1+\sup_{r \leq T}|\overline \xi_r|^{\eta}+|v|^{\eta}\right).
\end{equation}
Finally applying the last two inequalities to \eqref{eq:proof_lemma_inequality_for_H} and then going back to \eqref{eq:first_change_variable_product_integrals} we have the thesis.
\end{proof}
\section{Order of convergence}\label{sec:order}
In this section we finally investigate the order of convergence. We first prove the following proposition which will be crucial in the derivation of the order of convergence. 
\begin{proposition}\label{crucial_lemma_2}
There exist $C=C(T)>0$, $\iota>0$ such that
$$\mathbb{E}\left[ \sup_{0 \leq t \leq \tau} \left| \int_{0}^t e^{A_1 (t-s)} \left(F(U_{s},V_s^\epsilon)-\overline{F}(U_s)\right) ds \right|^2 \right]\leq C \epsilon (1+|u|^{\iota}+|v|^{\iota}),$$
for every $0 \leq  \tau \leq T$, $\epsilon>0$, $u \in H,v \in K$.
\end{proposition}
\begin{proof} We proceed with similar techniques to the ones in the proof of \cite[Theorem 4.1]{Cerrai_normal_dev}.\\
Define 
$$Z_t^\epsilon := \int_{0}^t e^{A_1 (t-s)} \left( F(U_{s},V_s^\epsilon)-\overline{F}(U_s)\right) ds.$$
We proceed using the factorization method, e.g. see \cite[Chapter 5, Section 3]{daprato_stochastic_eq}: fix $\zeta>0,n \geq 2$ integer and $ 1/(2n)< \beta < 1/3$ as in Hypothesis \ref{hp:operator1}, then:
\begin{align*}
Z_t^\epsilon = \frac{\sin{\beta \pi}}{\pi} \int_{0}^{t}   (t -s)^{\beta-1} e^{A_1 ( t -s) }   Y_s^{\epsilon} ds,
\end{align*}
where
\begin{align*}
Y_s^{\epsilon}=\int_{0}^{ s} (s-r)^{- \beta} e^{A_1(s-r) } \left( F(U_{r},V_r^\epsilon)-\overline{F}(U_r)\right)dr.
\end{align*}
By Holder's inequality  we have:
\begin{align}\label{eq:proof_crucial_lemma_est_sup_Z_with_Y}
\mathbb E \left[ \sup_{t \in [0,\tau]} |Z_t^\epsilon|^{2n} \right] 
& \leq     C  \int_0^\tau \mathbb E |Y_s^{\epsilon}|^{2n}  ds.
\end{align}
We now claim that there exist $C=C(T)>0$ and $\eta>0$ such that
\begin{equation}\label{eq:bound_expectation_Y_epsilon}
\mathbb{E}\left|Y^{\epsilon}_s\right|^{2 n}\leq C \epsilon^{n}\left(1+|u|^\eta+|v|^\eta\right) ,
\end{equation}
for every $0 \leq t_0 \leq  s \leq T$, $\epsilon>0$, $u \in H,v \in K$.\\
Indeed first recall the spectral representation \eqref{eq:spectral_representation_A_1}.
Then by Parseval's identity,  Holder's inequality, Hypothesis \ref{hp:operator1} and the properties of conditional expectations we have:
\begin{small}
\begin{align*}
\mathbb{E} \left|Y^{\epsilon}_s\right|^{2 n}=& \mathbb{E} \left(\sum_{k=1}^{\infty} \alpha_k^{-\frac{(n-1)\zeta}{ n}}  \alpha_k^{\frac{(n-1)\zeta}{ n}} \left|\int_{0}^{s}(s-r)^{-\beta} e^{-(s-r) \alpha_{k}}\langle F(U_{r},V_r^\epsilon)-\overline{F}(U_r), e_{k}\rangle d r\right|^{2}\right)^{n} \nonumber \\ 
& \leq \left(\sum_{k=1}^{\infty} \alpha_{k}^{-\zeta}\right)^{n-1} \sum_{k=1}^{\infty} \Big [ \alpha_{k}^{(n-1)\zeta}  \mathbb{E} \left|\int_{0}^{s}(s-r)^{-\beta} e^{-(s-r) \alpha_{k}}\langle F(U_{r},V_r^\epsilon)-\overline{F}(U_r), e_{k}\rangle d r\right|^{2n} \Bigg ] \nonumber  \\
& \leq C \sum_{k=1}^{\infty} \Big [ \alpha_{k}^{(n-1)\zeta}  \mathbb{E} \left|\int_{0}^{s}(s-r)^{-\beta} e^{-(s-r) \alpha_{k}}\langle F(U_{r},V_r^\epsilon)-\overline{F}(U_r), e_{k}\rangle d r\right|^{2n} \Bigg ] \nonumber  \\
& =  C\sum_{k=1}^{\infty} \Big [ \alpha_{k}^{(n-1)\zeta}  \mathbb{E}\int_{[0, s]^{2 n}} \prod_{i=1}^{2 n}\left(s-r_{i}\right)^{-\beta} e^{-\left(s-r_{i}\right) \alpha_{k}} \langle F(U_{r_i},V_{r_i}^\epsilon)-\overline{F}(U_{r_i}), e_{k} \rangle 
d r_{1} \cdots d r_{2 n} \Big ] \nonumber 
\\
& =  C\sum_{k=1}^{\infty} \Bigg [ \alpha_{k}^{(n-1)\zeta}  \mathbb{E} \Bigg[  \mathbb{E} \Bigg[   \int_{[0, s]^{2 n}} \prod_{i=1}^{2 n}\left(s-r_{i}\right)^{-\beta} e^{-\left(s-r_{i}\right) \alpha_{k}} \langle F(U_{r_i},V_{r_i}^\epsilon)-\overline{F}(U_{r_i}), e_{k} \rangle  \\
& \quad \quad \quad \quad \quad \quad \quad \quad \quad \quad \quad \quad \quad \quad
d r_{1} \cdots d r_{2 n}  \Bigg | \left \{U_r,  \forall r \leq s \right \}\Bigg ]  \Bigg ]  \Bigg ].
\end{align*}
\end{small}
Now consider the conditional expectation on the right-hand-side.
Fixing $s \geq 0$, due to the independence of the averaged component $\{U_r,  \forall r \leq s \}$ (which is $\sigma  \{ W_r^{Q_1}, \forall r \leq s  \} -$measurable) and the fast component $\{V_r^\epsilon,  \forall r \leq s \}$ (which is $\sigma \{ W_r^{Q_2}, \forall r \leq s  \}- $measurable being independent of $\sigma \{ W_r^{Q_1}, \forall r \leq s  \} $), for every $\overline  \xi \in C([0,s]; H)$ we have 
\begin{small}
\begin{align*}
\mathbb{E}  \Bigg[   \int_{[0, s]^{2 n}}& \prod_{i=1}^{2 n}\left(s-r_{i}\right)^{-\beta} e^{-\left(s-r_{i}\right) \alpha_{k}} \langle F(U_{r_i},V_{r_i}^\epsilon)-\overline{F}(U_{r_i}), e_{k} \rangle 
d r_{1} \cdots d r_{2 n}  \Bigg | \{U_r,  \forall r \leq s \}= \{\overline  \xi_r,  \forall r \leq s \} \Bigg ] \\
& = \mathbb{E} \Bigg[   \int_{[0, s]^{2 n}} \prod_{i=1}^{2 n}\left(s-r_{i}\right)^{-\beta} e^{-\left(s-r_{i}\right) \alpha_{k}} \langle F(\overline  \xi_{r_i},V_{r_i}^\epsilon)-\overline{F}(\overline  \xi_{r_i}), e_{k} \rangle 
d r_{1} \cdots d r_{2 n}  \Bigg ]\\
& \leq   \int_{[0, s]^{2 n}} \Bigg | \mathbb{E} \Bigg[  \prod_{i=1}^{2 n}\left(s-r_{i}\right)^{-\beta} e^{-\left(s-r_{i}\right) \alpha_{k}} \langle F(\overline  \xi_{r_i},V_{r_i}^\epsilon)-\overline{F}(\overline  \xi_{r_i}), e_{k} \rangle \Bigg ]  \Bigg  |
d r_{1} \cdots d r_{2 n}  \\
& \leq C \left(1+\sup_{r \leq T}|\overline \xi_r|^{\eta}+|v|^{\eta}\right) \epsilon^{n} \left( \frac{1}{\alpha_k} \right)^{(1-2\beta)n},
\end{align*}
\end{small}
where the last inequality follows by Lemma \ref{lemma:crucial_lemma}.\\
Hence we have:
\begin{align*}
\mathbb{E} \Bigg[   \int_{[0, s]^{2 n}} & \prod_{i=1}^{2 n}\left(s-r_{i}\right)^{-\beta} e^{-\left(s-r_{i}\right) \alpha_{k}} \langle F(U_{r_i},V_{r_i}^\epsilon)-\overline{F}(U_{r_i}), e_{k} \rangle 
d r_{1} \cdots d r_{2 n}  \Bigg | \{U_r,  \forall r \leq s \}\Bigg ] \\
&  \leq C \left(1+\sup_{r \leq T}|U_r|^{\eta}+|v|^{\eta}\right) \epsilon^{n} \left( \frac{1}{\alpha_k} \right)^{(1-2\beta)n}.
\end{align*}
Inserting this inequality in the one for $\mathbb{E} \left|Y^{\epsilon}_s\right|^{2 n}$ we have:
\begin{align*}
\mathbb{E}\left|Y^{\epsilon}_s\right|^{2 n} & \leq C  \epsilon^{n}\left(1+\mathbb E \left [\sup_{r \leq T}|U_{r}|^\eta\right]+|v|^\eta\right) \sum_{k=1}^{\infty} \alpha_{k}^{(n-1)\zeta-(1-2 \beta)n}\\
& \leq C   \epsilon^{n}\left(1+|u|^\eta+|v|^\eta\right) \sum_{k=1}^{\infty} \alpha_{k}^{n(\zeta+2\beta-1)-\zeta}.
\end{align*}
The series on the right-hand-side is convergent by Hypothesis \ref{hp:operator1} and we have \eqref{eq:bound_expectation_Y_epsilon}, so that the claim is proved.

Inserting \eqref{eq:bound_expectation_Y_epsilon} into \eqref{eq:proof_crucial_lemma_est_sup_Z_with_Y} we have:
\begin{align*}
\mathbb E \left[ \sup_{t \in [0,\tau]} |Z_t^\epsilon|^{2n} \right] 
\leq C   \epsilon^{n} \left(1+|u|^\eta+|v|^\eta\right).
\end{align*}
Finally  by Holder's inequality we have the thesis:
\begin{align*}
\mathbb{E}\left[ \sup_{t \in [0,\tau]} \left|Z_t^\epsilon \right|^{2 } \right] & \leq  \left (\mathbb{E}\left[ \sup_{t \in [0,\tau]} \left|Z_t^\epsilon \right|^{2 n} \right] \right)^{1/n}  \leq C  \epsilon\left(1+|u|^\iota+|v|^\iota\right),
\end{align*}
for $\iota=\eta / n$.
\end{proof}
We can now state and prove the main Theorem of this work:
\begin{theorem}\label{th:main_theorem}
Let $u \in H$, $v \in K$ and assume Hypotheses \ref{hp:operator1}, \ref{hp:operator2}, \ref{hp:coefficients}, \ref{hp:hilbert_schmidt}, \ref{hp:covariance_inverse}. Then there exists $C=C(T,|u|,|v|)>0$ such that
\begin{equation*}
\mathbb{E}\left [ \sup_{t\in [0,T]} \abs{U^\epsilon_t-U_t}^2 \right] \leq C \epsilon ,
\end{equation*}
for every $\epsilon>0$.
\end{theorem}
\begin{proof}
For $t \in [0,T]$ we have:
\begin{align*}
U_t^\epsilon - U_t  =  \int_{0}^t e^{A_1 (t-s)} \left( F(U_{s}^\epsilon,V_s^\epsilon)-\overline{F}(U_s)\right) ds,
\end{align*}
so that: 
\begin{align*}
\left | U_t^\epsilon - U_t \right| & \leq \left| \int_{0}^t e^{A_1 (t-s)} \left( F(U_{s}^\epsilon,V_s^\epsilon)-F(U_{s},V_s^\epsilon)\right) ds \right| +\left| \int_{0}^t e^{A_1 (t-s)} \left( F(U_{s},V_s^\epsilon)-\overline{F}(U_s)\right) ds \right|.
\end{align*}
Now let $0 \leq \tau \leq T$ and compute
\begin{align*}
\mathbb{E}\left[ \sup_{0 \leq t \leq \tau} \left | U_t^\epsilon - U_t \right|^2 \right] 
& \leq  2 \mathbb{E}\left[ \sup_{0 \leq t \leq \tau} \left| \int_{0}^t e^{A_1 (t-s)} \left( F(U_{s}^\epsilon,V_s^\epsilon)-F(U_{s},V_s^\epsilon)\right) ds \right|^2 \right] \\
&   + 2 \mathbb{E}\left[ \sup_{0 \leq t \leq \tau}\left| \int_{0}^t e^{A_1 (t-s)} \left( F(U_{s},V_s^\epsilon)-\overline{F}(U_s)\right) ds \right|^2 \right].
\end{align*}
For the first term on the right-hand-side by the Lipschitz continuity of $F$ we have:
\begin{align*}
\mathbb{E}\left[ \sup_{0 \leq t \leq \tau}\left| \int_{0}^t e^{A_1 (t-s)} \left( F(U_{s}^\epsilon,V_s^\epsilon)-F(U_{s},V_s^\epsilon)\right) ds \right|^2 \right] & \leq C \int_{0}^\tau \mathbb{E}  \left| U_s^{\epsilon}-U_s\right|^2 ds . 
\end{align*}
For the second term on the right-hand-side by Proposition \ref{crucial_lemma_2} we have:
\begin{align*}
 \mathbb{E}\left[ \sup_{0 \leq t \leq \tau} \left| \int_{0}^t e^{A_1 (t-s)} \left( F(U_{s},V_s^\epsilon)-\overline{F}(U_s)\right) ds \right|^2 \right]
 \leq C \epsilon.
\end{align*}
Putting everything together we have:
\begin{align}\label{eq:proof_eq_for_gronwall_order_conv}
\mathbb{E}\left[ \sup_{0 \leq t \leq \tau} \left |  U_t^\epsilon - U_t \right|^2 \right] \leq C \left (\epsilon + \int_{0}^\tau \mathbb{E}  \left| U_s^{\epsilon}-U_s\right|^2 ds  \right)  ,
\end{align}
for every $\tau \leq  T$.\\
Then by Gronwall's Lemma we have the thesis of the Theorem:
\begin{align*}
\mathbb{E}\left[ \sup_{0 \leq t \leq T} \left | U_t^\epsilon - U_t \right|^2 \right] \leq C \epsilon.
\end{align*}
\end{proof}
Finally we can provide an application to which our theory can be applied and which is not covered by the existing literature.
\begin{example}
Consider the following fully coupled slow-fast stochastic reaction-diffusion system:
\begin{equation}\label{eq:spde}
\begin{cases}
\frac{\partial u_{\varepsilon}}{\partial t}(t, \xi)= \frac{\partial^2}{\partial \xi^2} u_{\varepsilon}(t, \xi)+f\left(\xi, u_{\varepsilon}(t, \xi), v_{\varepsilon}(t, \xi)\right) + \dot {w}_1(t, \xi) ,\\
 \frac{\partial v_{\varepsilon}}{\partial t}(t, \xi)= \frac{1}{\varepsilon}\left[\left(\frac{\partial^2}{\partial \xi^2} - \lambda \right) v_{\varepsilon}(t, \xi)+g\left(\xi, v_{\varepsilon}(t, \xi)\right)\right] +\frac{1}{\sqrt{\varepsilon}} \dot {w}_2(t, \xi) ,\\  
 u_{\varepsilon}(0, \xi)=u (\xi), \quad v_{\varepsilon}(0, \xi)=v(\xi), \quad \xi \in [0,L] ,\\
 u_{\varepsilon}(t, \xi)=  v_{\varepsilon}(t, \xi)=0, \quad t \geq 0, \quad \xi=0,L,
\end{cases}
\end{equation}
where 
\begin{itemize}
\item $t \in [0,T], \xi \in [0,L]$,
\item  $\epsilon \in (0,1]$ is a small parameter representing the ratio of time-scales between the two variables of the system $u_\epsilon$ and $v_\epsilon$,
\item $u_\epsilon$ and $v_\epsilon$ are the slow and fast components respectively,
\item $u,v	\in  H= L^2[0,T]$ are the initial conditions,
\item $\lambda>0$,
\item $f,g \colon [0,L] \times \mathbb{R} \rightarrow  \mathbb{R}$ are Lipschitz functions uniformly wrt $\xi$ with Lipschitz constants $L_f,L_G$ respectively and $L_G < \lambda$,
\item $\dot {w}_1, \dot {w}_2$ are independent white noises both in time and space.
\end{itemize}
Then it is well known \cite{Cerrai} that \eqref{eq:spde} can be rewritten in the abstract form \eqref{eq:abstract_system} where $H=K=L^2[0,T]$, $F\colon H\times H\rightarrow  H$, $G\colon H \rightarrow  H$ are the Nemytskii operators of $f,g$ respectively, i.e.
\begin{align*}
&F(x,y)(\xi)=f(\xi,x(\xi),y(\xi)), \\
&G(y)(\xi)=g(\xi,y(\xi)).
\end{align*}
In this setting the hypotheses of Theorem \ref{th:main_theorem} are satisfied (recall Remarks \ref{rem:laplacian_hp_1}, \ref{rem:covariances_laplacian}) so that the result can be applied.
\end{example}
\appendix

\section{Proof of Lemma \ref{lemma:mixing_1}}
\begin{proof}
First observe that $$\mathcal{H}_s^{t}(v)=\sigma(\mathcal{C}_s^t(v),0 \leq s \leq t),$$
where 
\begin{equation}\label{eq:proof_cylindrical_sets}
\mathcal{C}_s^t(v)=\{v_{r_1}^{v} \in A_1,...v_{r_k}^{v} \in A_k : k \in \mathbb{N}, s \leq r_1< ... < r_k \leq t, A_1, ... A_k \in \mathcal{B}(K) \}
\end{equation}
is the family of cylindrical sets.\\
Consider $B_{1} \in \mathcal{C}_{0}^{t}(v)$ and $B_{2} \in \mathcal{C}_{s+t}^{\infty}(v)$, i.e.
$$
B_{1}=\bigcap_{i=1}^{k_{1}}\left\{v^{v}_{r_{1, i}} \in A_{1, i}\right\}, \quad B_{2}=\bigcap_{i=1}^{k_{2}}\left\{v^{v}_{r_{2, i}} \in A_{2, i}\right\},
$$
for $0 \leq r_{1,1}< \cdots <r_{1, k_{1}} \leq t$ and $s+t \leq r_{2,1}<\cdots<r_{2, k_{2}}<\infty$ and $A_{j, i} \in \mathcal{B}(K)$, for $j=1,2$ and $i=1, \ldots, k_{j}$.\\ 
First by the tower property of conditional expectations we have:
\begin{align}\label{eq:proof_probability_intersection}
\mathbb{P}\left(B_{1} \cap B_{2}\right) &=\mathbb{E}\left[\prod_{i=1}^{k_{1}} 1_{A_{1, i}}\left(v^{v}_{r_{1, i}}\right) \prod_{i=1}^{k_{2}} 1_{A_{2, i}}\left(v^{v}_{r_{2, i}}\right)\right] \nonumber \\
&=\mathbb{E}\left[ \prod_{i=1}^{k_{1}} 1_{A_{1, i}}\left(v^{v}_{r_{1, i}}\right) \mathbb{E}\left[\prod_{i=1}^{k_{2}} 1_{A_{2, i}}\left(v^{v}_{r_{2, i}}\right) \bigg | \mathcal{F}_{t}\right]\right],
\end{align}
where $1_{A_{i,j}}(\cdot)$ is the indicator function.\\
Now, as $t+s \leqslant r_{2,1}<\cdots<r_{2, k_{2}}$, we have:
$$
\begin{aligned}
\mathbb{E}\left[ \prod_{i=1}^{k_{2}} 1_{A_{2, i}}\left(v^{v}_{r_{2, i}}\right) \bigg | \mathcal{F}_{t}\right] &=\mathbb{E}\left[1_{A_{2,1}}\left(v^{v}_{r_{2,1}}\right) \mathbb{E}\left[\prod_{i=2}^{k_{2}} 1_{A_{2, i}}\left(v^{v}_{r_{2, i}}\right)  \bigg |  \mathcal{F}_{r_{2,1}}\right]  \bigg | \mathcal{F}_{t}\right] \\
&=\mathbb{E}\left[1_{A_{2,1}}\left(v^{v}_{r_{2,1}}\right) \mathbb{E}\left[1_{A_{2,2}}\left(v^{v}_{r_{2,2}}\right) \mathbb{E}\left[\prod_{i=3}^{k_{2}} 1_{A_{2, i}}\left(v^{v}_{r_{2, i}}\right)  \bigg |  \mathcal{F}_{r_{2,2}}\right]  \bigg |  \mathcal{F}_{r_{2,1}}\right]  \bigg | \mathcal{F}_{t}\right].
\end{aligned}
$$
By iteration we have:
\begin{small}
\begin{align*}
\mathbb{E}\left[\prod_{i=1}^{k_{2}} 1_{A_{2, i}}\left( v^{v}_{r_{2, i}}\right) \bigg | \mathcal{F}_{t}\right]& =\mathbb{E}\left[   1_{A_{2,1}}\left(v^{v}_{r_{2,1}}\right) \mathbb{E}\left[ 1_{A_{2,2}}\left(v^{v}_{r_{2,2}}\right)\mathbb{E}\left[1_{A_{2,3}}\left(v^{v}_{r_{2,3}}\right) \cdots \mathbb{E}\left[1_{A_{2, k_{2}}}\left(v^{v}_{r_{2, k_{2}}}\right) \mid \mathcal{F}_{r_{2, k_{2}-1}} \right]|\cdots | \mathcal{F}_{r_{2,2}} \right]  | \mathcal{F}_{r_{2,1}} \right] \mid \mathcal{F}_{t}    \right ]\\
& =P_{r_{2,1}-t}\left[1_{A_{2,1}} P_{r_{2,2}-r_{2,1}}\left(1_{A_{2,2}} P_{r_{2,3}-r_{2,2}}\left( \cdots I_{A_{2,k_2-1}} P_{r_{2,k_2}-r_{2,k_2-1}}\left ( I_{A_{2,k_2}} \right) \cdots \right)\right)\right]\left(v^{v}_t\right),
\end{align*}
\end{small}
so that by \eqref{eq:proof_probability_intersection} we have:
\begin{small}
\begin{align*}
P(B_1 \cap B_2)= \mathbb{E}\left[ \prod_{i=1}^{k_{1}} 1_{A_{1, i}}\left(v^{v}_{r_{1, i}}\right) P_{r_{2,1}-t}\left[1_{A_{2,1}} P_{r_{2,2}-r_{2,1}}\left(1_{A_{2,2}} P_{r_{2,3}-r_{2,2}}\left( \cdots I_{A_{2,k_2-1}} P_{r_{2,k_2}-r_{2,k_2-1}}\left ( I_{A_{2,k_2}} \right) \cdots \right)\right)\right]\left(v^{v}_t\right)\right].
\end{align*}
\end{small}
In a similar way we have:
$$
\mathbb{E} \prod_{i=1}^{k_{2}} 1_{A_{2, i}}\left(v^{v}\left(r_{2, i}\right)\right)=P_{r_{2,1}-t}\left[1_{A_{2,1}} P_{r_{2,2}-r_{2,1}}\left(1_{A_{2,2}} P_{r_{2,3}-r_{2,2}}\left( \cdots I_{A_{2,k_2-1}} P_{r_{2,k_2}-r_{2,k_2-1}}\left ( I_{A_{2,k_2}} \right) \cdots \right)\right)\right]\left(v\right),
$$
so that
\begin{small}
\begin{align*}
P(B_1)P(B_2)&=\mathbb{E} \prod_{i=1}^{k_{1}} 1_{A_{1, i}}\left(v^{v}_{r_{1, i}}\right) \mathbb{E} \prod_{i=1}^{k_{2}} 1_{A_{2, i}}\left(v^{v}\left(r_{2, i}\right)\right)\\
& =\mathbb{E} \prod_{i=1}^{k_{1}} 1_{A_{1, i}}\left(v^{v}_{r_{1, i}}\right) P_{r_{2,1}-t}\left[1_{A_{2,1}} P_{r_{2,2}-r_{2,1}}\left(1_{A_{2,2}} P_{r_{2,3}-r_{2,2}}\left( \cdots I_{A_{2,k_2-1}} P_{r_{2,k_2}-r_{2,k_2-1}}\left ( I_{A_{2,k_2}} \right) \cdots \right)\right)\right]\left(v\right).
\end{align*}
\end{small}
It follows
\begin{small}
\begin{align*}
\mathbb{P}\left(B_{1} \cap B_{2}\right) &-\mathbb{P}\left(B_{1}\right) \mathbb{P}\left(B_{2}\right)\\
&=\mathbb{E}  \Bigg [ \prod_{i=1}^{k_{1}} 1_{A_{1, i}}\left(v^{u,v}\left(r_{1, i}\right)\right) \Bigg \{ P_{r_{2,1}-t}\left[1_{A_{2,1}} P_{r_{2,2}-r_{2,1}}\left(1_{A_{2,2}} P_{r_{2,3}-r_{2,2}}\left( \cdots I_{A_{2,k_2-1}} P_{r_{2,k_2}-r_{2,k_2-1}}\left ( I_{A_{2,k_2}} \right) \cdots \right)\right)\right]\left(v^{v}_t\right)\\
& \quad \quad   - P_{r_{2,1}-t}\left[1_{A_{2,1}} P_{r_{2,2}-r_{2,1}}\left(1_{A_{2,2}} P_{r_{2,3}-r_{2,2}}\left( \cdots I_{A_{2,k_2-1}} P_{r_{2,k_2}-r_{2,k_2-1}}\left ( I_{A_{2,k_2}} \right) \cdots \right)\right)\right]\left(v\right) \Bigg \}  \Bigg ] \\
& = \mathbb{E}  \Bigg [ \prod_{i=1}^{k_{1}} 1_{A_{1, i}}\left(v^{v}\left(r_{1, i}\right)\right) \Bigg \{ P_{r_{2,1}-t}\phi \left(v^{v}_t\right)  -  P_{r_{2,1}-t} \phi(v) \Bigg \}  \Bigg ] \\
& = \mathbb{E}  \Bigg [ \prod_{i=1}^{k_{1}} 1_{A_{1, i}}\left(v^{v}\left(r_{1, i}\right)\right) \Bigg \{ P_{r_{2,1}-t}\phi \left(v^{v}_t\right)  - \int_K \phi(v) \mu(dv) \Bigg \}  \Bigg ] ,
\end{align*}
\end{small}
for $\phi \coloneq 1_{A_{2,1}} P_{r_{2,2}-r_{2,1}}\left(1_{A_{2,2}} P_{r_{2,3}-r_{2,2}}\left( \cdots I_{A_{2,k_2-1}} P_{r_{2,k_2}-r_{2,k_2-1}}\left ( I_{A_{2,k_2}} \right) \cdots \right)\right)$.\\
 Then by Lemmas \ref{lemma:moments_fast_motion} and \ref{lemma:convergence_equilibrium} and as $f(s)=e^{-\delta s} s^{-1/2}$ is decreasing we have:
\begin{align*}
\left|\mathbb{P}\left(B_{1} \cap B_{2}\right)-\mathbb{P}\left(B_{1}\right) \mathbb{P}\left(B_{2}\right)\right| &\leq C \frac{ e^{-\delta\left(r_{2,1}-t\right)}}{\sqrt{r_{2,1}-t}}\left(1+\mathbb{E}\left|v^{v}_t\right|+|v|\right)\\ & \leq C \frac{ e^{-\delta s}}{\sqrt{s}}\left(1+|v|\right),
\end{align*}
so that the inequality holds when $B_{1} \in \mathcal{C}_{0}^{t}(v)$ and $B_{2} \in \mathcal{C}_{t+s}^{\infty}(v)$.\\
Finally recalling \eqref{eq:proof_cylindrical_sets} the validity of the inequality can be extended to every $B_{1} \in \mathcal{H}_{0}^{t}(v)$ and $B_{2} \in \mathcal{H}_{s+t}^{\infty}(v)$.
\end{proof}
\section{Proof of Lemma \ref{lemma:mixing_2}}
\begin{proof}
As $|\xi_1| \leq 1$ a.s. we have:
\begin{align*}
\left|\mathbb{E} \left[ \xi_{1} \xi_{2} \right] -\mathbb{E} \xi_{1} \mathbb{E} \xi_{2}\right| & =\left|\mathbb{E}\left[\xi_{1} \mathbb{E}\left[\xi_{2} \mid \mathcal{H}_{s_{1}}^{t_{1}}(v)\right]\right]-\mathbb{E} \xi_{1} \mathbb{E} \xi_{2}\right| \\
& =\left|\mathbb{E}\left[\xi_{1}\left ( \mathbb{E}\left[\xi_{2} \mid \mathcal{H}_{s_{1}}^{t_{1}}(v)\right]-\mathbb{E} \xi_{2}\right )\right]\right| \\
& \leq\left|\mathbb{E}\left[\tilde \xi_{1}\left (\mathbb{E}\left[\xi_{2} \mid \mathcal{H}_{s_{1}}^{t_{1}}(v)\right]-\mathbb{E} \xi_{2}\right )\right]\right|\\
 & =\left|\mathbb{E} \left[  \tilde \xi_{1} \xi_{2} \right] -\mathbb{E} \tilde \xi_{1} \mathbb{E} \xi_{2}\right|,
\end{align*}
where we have defined
$$
\tilde \xi_{1} \coloneq 
\begin{cases} 
1, \quad & \textit{if} \quad  \mathbb{E}\left[\xi_{2} \mid \mathcal{H}_{s_{1}}^{t_{1}}(v)\right]-\mathbb{E} \xi_{2}>0,\\
-1, \quad &\textit{if} \quad   \mathbb{E}\left[\xi_{2} \mid \mathcal{H}_{s_{1}}^{t_{1}}(v)\right]-\mathbb{E} \xi_{2} \leq 0.
\end{cases}
$$
Similarly, as $|\xi_2| \leq 1$ a.s., we have:
\begin{align*}
\left|\mathbb{E} [ \tilde \xi_{1} \xi_{2} ] -\mathbb{E} \tilde \xi_{1} \mathbb{E} \xi_{2}\right|& =\left|\mathbb{E}\left[ \xi_{2}   \mathbb{E}\left[\tilde \xi_{1} \mid \mathcal{H}_{s_{2}}^{t_{2}}(v)\right]\right]-\mathbb{E} \tilde \xi_{1} \mathbb{E} \xi_{2}\right| \\ 
& =\left|\mathbb{E}\left[\xi_{2}\left ( \mathbb{E}\left[\tilde \xi_{1} \mid \mathcal{H}_{s_{1}}^{t_{1}}(v)\right]-\mathbb{E}\tilde  \xi_{1}\right )\right]\right| \\
& \leq \left|\mathbb{E}\left[\tilde \xi_{2}\left ( \mathbb{E}\left[\tilde \xi_{1} \mid \mathcal{H}_{s_{1}}^{t_{1}}(v)\right]-\mathbb{E}\tilde  \xi_{1}\right )\right]\right| \\
& =\left|\mathbb{E}[\tilde \xi_{1} \tilde \xi_{2} ]-\mathbb{E} \tilde \xi_{1}\mathbb{E} \tilde \xi_{2}\right|,
\end{align*}
where we have defined
$$
\tilde \xi_{2} \coloneq 
\begin{cases} 
1, \quad & \textit{if} \quad  \mathbb{E}\left [\tilde \xi_{1} \mid \mathcal{H}_{s_{2}}^{t_{2}}(v)\right ]-\mathbb{E} \tilde \xi_{1}>0,\\
-1, \quad & \textit{if} \quad  \mathbb{E}\left[\tilde \xi_{1} \mid \mathcal{H}_{s_{2}}^{t_{2}}(v)\right]-\mathbb{E} \tilde \xi_{1} \leq 0.
\end{cases}
$$
Then it follows:
\begin{align*}
\left|\mathbb{E} [ \xi_{1} \xi_{2} ]-\mathbb{E} \xi_{1} \mathbb{E} \xi_{2}\right| & \leq \left|\mathbb{E} [\tilde \xi_{1}\tilde \xi_{2}]-\mathbb{E} \tilde \xi_{1}\mathbb{E} \tilde \xi_{2}\right|.
\end{align*}
Notice that 
\begin{align*}
& \tilde \xi_{1}=2 1_{A}-1, \quad  \tilde \xi_{2}= 2 1_{B}-1
\end{align*}
with 
\begin{align*}
& A=\left\{\omega \in \Omega : \mathbb{E}\left [\xi_{2} \mid \mathcal{H}_{s_{1}}^{t_{1}}(v)\right]-\mathbb{E} \xi_{2}>0\right\},\\
& B=\left\{\omega \in \Omega: \mathbb{E}\left[\tilde \xi_{1}\mid \mathcal{H}_{s_{2}}^{t_{2}}(v)\right]-\mathbb{E} \tilde \xi_{1}>0\right\}.
\end{align*}
It follows:
\begin{align*}
\left|\mathbb{E} [ \xi_{1} \xi_{2} ]-\mathbb{E} \xi_{1} \mathbb{E} \xi_{2}\right| & \leq \left|\mathbb{E} [\tilde \xi_{1}\tilde \xi_{2}]-\mathbb{E} \tilde \xi_{1}\mathbb{E} \tilde \xi_{2}\right|
= |\mathbb E[(2 1_A-1)(2 1_B-1)]-\mathbb E[2 1_A-1]\mathbb E[2 1_B-1]| \\
&= 4|\mathbb E[1_A 1_B]-\mathbb E[ 1_A]\mathbb E[1_B]|= 4|\mathbb{P}(A \cap B)-\mathbb{P}(A) \mathbb{P}(B)|.
\end{align*}
Now by Lemma \ref{lemma:mixing_1} as $A \in \mathcal{H}_{s_{1}}^{t_{1}}(v), B \in \mathcal{H}_{s_{2}}^{t_{2}}(v)$ we have:
$$
\left|\mathbb{E} [\xi_{1} \xi_{2}]-\mathbb{E} \xi_{1} \mathbb{E} \xi_{2}\right| \leq C \frac{e^{-\delta\left(s_{2}-t_{1}\right)}}{\sqrt{s_{2}-t_{1}}}\left(1+|v|\right)
$$
so that we have the thesis of the Lemma.
\end{proof}
\section{Proof of Lemma \ref{lemma:mixing_3}}
\begin{proof}
We proceed in a similar way to the proof of \cite[Proposition 3.3]{Cerrai_normal_dev}.\\
Indeed fix $R>0$ and set
$A_{1, R}=\left\{\omega \in \Omega : \left|\xi_{1}\right| \leq R\right\}$, $A_{2, R}=\left\{\omega \in \Omega:\left|\xi_{2}\right| \leq R\right\}$.\\
Then we have:
\begin{align*}
\mathbb{E} [\xi_{1} \xi_{2}]-\mathbb{E} \xi_{1} \mathbb{E} \xi_{2}=& \mathbb{E}\left[\xi_{1} \xi_{2} \left (1_{A_{1, R} \cap A_{2, R}}+ 1_{ A_{1, R}^{c} \cup A_{2, R}^{c}} \right )\right]-\mathbb{E} \xi_{1} \mathbb{E} \xi_{2}  \\
=&\left\{\mathbb{E} [ \xi_{1} 1_{A_{1, R}} \xi_{2} 1_{A_{2, R}} ]-\mathbb{E} [ \xi_{1} 1_{A_{1, R}}] \mathbb{E} [\xi_{2} 1_{A_{2, R}}]\right \}+\mathbb{E} [\xi_{1} \xi_{2} 1_{A_{1, R}^{c} \cup A_{2, R}^{c}}] \nonumber \\
&-\left\{\mathbb{E} [\xi_{1} 1_{A_{1, R}} ]\mathbb{E}[ \xi_{2} 1_{A_{2, R}^{c}}]+\mathbb{E}[ \xi_{1} 1_{A_{1, R}^{c}}] \mathbb{E} [\xi_{2} 1_{A_{2, R}}]+\mathbb{E}[ \xi_{1} 1_{A_{1, R}^{c}}] \mathbb{E} [\xi_{2} 1_{A_{2, R}^{c}}]\right\}  \\
=& T_{1, R}+T_{2, R}+T_{3, R} .
\end{align*}
Consider $T_{1, R}$, then we have:
$$
T_{1, R}=R^{2}\left(\mathbb{E} \left [ \frac{\xi_{1}}{R} 1_{A_{1, R}} \frac{\xi_{2}}{R} 1_{A_{2, R}} \right]-\mathbb{E} \left [\frac{\xi_{1}}{R} 1_{A_{1, R}} \right ]\mathbb{E} \left[\frac{\xi_{2}}{R} 1_{A_{2, R}} \right]\right).
$$
Now as $\left|\frac{\xi_{i}}{R} 1_{A_{i, R}}\right| \leq 1$ a.s then by Lemma \ref{lemma:mixing_2} we have:
$$
\left|T_{1, R}\right| \leq C  R^{2} \frac{e^{-\delta\left(s_{2}-t_{1}\right)}}{\sqrt{s_{2}-t_{1}}} \left(1+|v|\right). 
$$
For $T_{2, R}$ by Holder's and Markov's inequalities we have:
\begin{align*}
\left|T_{2, R}\right|^{\frac{2}{1-\rho}} & \leq \mathbb{E}\left|\xi_{1}\right|^{\frac{2}{1-\rho}} \mathbb{E}\left|\xi_{2}\right|^{\frac{2}{1-\rho}}\left(\mathbb{P}\left(A_{1, R}^{c}\right)+\mathbb{P}\left(A_{2, R}^{c}\right)\right)^{\frac{2 \rho}{1-\rho}}\\
 & \leq \mathbb{E}\left|\xi_{1}\right|^{\frac{2}{1-\rho}} \mathbb{E}\left|\xi_{2}\right|^{\frac{2}{1-\rho}} R^{-\frac{2 \rho}{1-\rho}}\left(\mathbb{E}\left|\xi_{1}\right|+\mathbb{E}\left|\xi_{2}\right|\right)^{\frac{2 \rho}{1-\rho}}
\end{align*}
and then by \eqref{eq:lemma-hp_mixing_3} we have:
$$
\left|T_{2, R}\right| \leq C K_1 K_2 (K_1+K_2)^{\rho }R^{-\rho}.
$$
For the first term of $T_{3,R}$ we have:
$$|\mathbb{E} [\xi_{1} 1_{A_{1, R}} ]\mathbb{E}[ \xi_{2} 1_{A_{2, R}^{c}}] |\leq \left(\mathbb{E}|\xi_1|^{\frac{2}{1-\rho}} \right)^{\frac{1-\rho}{2}}\left(\mathbb{E}|\xi_2|^{\frac{2}{1-\rho}} \right)^{\frac{1-\rho}{2}}\mathbb{P}\left(A_{2, R}^{c}\right)^\rho.$$
Also the other terms can be treated in an analogous way and then similarly to before we have:
$$
\left|T_{3, R}\right| \leq C K_1 K_2 (K_1+K_2)^{\rho }R^{-\rho}.
$$
Now by inserting the inequalities for $T_{i,R}$ into the first equation we have:
$$
\left|\mathbb{E} \xi_{1} \xi_{2}-\mathbb{E} \xi_{1} \mathbb{E} \xi_{2}\right| \leq C \frac{e^{-\delta\left(s_{2}-t_{1}\right)}}{\sqrt{s_{2}-t_{1}}} \left(1+|v|\right)  R^{2}+C  K_1 K_2 (K_1+K_2)^{\rho } R^{-\rho}.
$$
By minimizing over $R>0$ the right-hand-side of the previous inequality we have:
$$
\left|\mathbb{E} \xi_{1} \xi_{2}-\mathbb{E} \xi_{1} \mathbb{E} \xi_{2}\right| \leq C K_1^{\frac{2}{2+\rho}} K_2^{\frac{2}{2+\rho}} (K_1+K_2)^{\frac{2 \rho}{2+\rho}}  \left(\frac{e^{-\delta ( s_{2}-t_{1})}}{\sqrt{s_{2}-t_{1}}} \left(1+|v|\right) \right)^{\frac{\rho}{2+\rho}} .
$$
\end{proof}
\section*{Acknowledgements}
The author thanks warmly his PhD supervisor Giuseppina Guatteri for her precious and constant support during the preparation of this manuscript. He also thanks Sandra Cerrai and Gianmario Tessitore for some relevant conversations related to the content of the present paper. Finally he thanks the two anonimous referees for their very careful reading of the manuscript and their precious remarks and suggestions.

\end{document}